\documentclass[10pt]{amsart}

\usepackage{amsmath, amsfonts, amssymb,mathtools}
\usepackage[a4paper, margin=1.4in]{geometry}

\usepackage{latexsym, verbatim}
\usepackage{graphicx}
\usepackage{caption}
\usepackage{color}
\usepackage{url}
\usepackage{mathdots}
\usepackage[table]{xcolor}
\usepackage{hyperref}
\usepackage{tikz}
\usetikzlibrary{cd}
\usepackage[all]{xy}
\usepackage{stackengine}
\usepackage{enumitem}

\theoremstyle{plain}
\newtheorem{thmintro}{Theorem}

\newtheorem{corintro}{Corollary}[thmintro]

\newtheorem{thm}{Theorem}[section]
 \newtheorem{cor}[thm]{Corollary}
 \newtheorem{lem}[thm]{Lemma}
 \newtheorem{prop}[thm]{Proposition}

\theoremstyle{definition}
 \newtheorem{ex}[thm]{Example}
 \newtheorem{rmk}[thm]{Remark}
 \newtheorem{ques}[thm]{Question}

 \newtheorem{defn}[thm]{Definition}

\newcommand{\II}{ \mathbb{I}}
\newcommand{\JJ}{ \mathbb{J}}

\newcommand{\Z}{\mathbb{Z}}
\newcommand{\R}{\mathbb{R}}
\newcommand{\CC}{\mathbb{C}}

\newcommand{\PP}{\mathbb{P}}

\newcommand{\cA}{{\mathcal A}}

\newcommand{\Hh}{{\mathcal H}}

\newcommand{\cM}{{\mathcal M}}

\newcommand{\cV}{{\mathcal V}}
\newcommand{\rank}{\operatorname{rank}}

\newcommand\isomto{\stackrel{\sim}{\smash{\longrightarrow}\rule{0pt}{0.4ex}}}

\newcommand{\ot}{\otimes}

\newcommand{\Imt}{{\rm Im}\,}
\newcommand{\Ke}{{\rm Ker}\,}
\newcommand{\Cok}{{\rm Coker}\,}

\newcommand{\del}{\partial}
\newcommand{\delb}{{\bar \partial}}
\newcommand{\mub}{{\bar \mu}}

\newcommand{\gr}{\operatorname{gr}}

\newcommand{\mult}{\operatorname{mult}}

\newcommand{\img}[2][1]{\begin{gathered}\includegraphics[scale=#1]{#2}\end{gathered}}

\newcommand{\Img}{\mathrm{Im}}

\newcommand{\dR}{\mathrm{dR}}

\newcommand{\id}{\mathrm{id}}
\newcommand{\pdef}{\mathrm{pdef}}

\title{A $dd^c$-type condition beyond the K\"ahler realm}

\author[J. Stelzig]{Jonas Stelzig}
  \address[J. Stelzig]{Mathematisches Institut der Ludwig-Maximilians-Universit\"at M\"unchen, Theresienstra{\ss}e 39, 80333 M\"unchen.}
  \email{jonas.stelzig@math.lmu.de}

\author[S. Wilson]{Scott O. Wilson}
  \address[S. Wilson]{Department of Mathematics, Queens College, City University of New York, 65-30 Kissena Blvd., Flushing, NY 11367}
  \email{scott.wilson@qc.cuny.edu}

\subjclass[2010]{}

\begin{document}

\begin{abstract}
This paper introduces a generalization of the $dd^c$-condition for complex manifolds.  Like the $dd^c$-condition, it admits a diverse collection of characterizations, and is hereditary under various geometric constructions. Most notably, it is an open property with respect to small deformations.  The condition is satisfied by a wide range of complex manifolds including all compact complex surfaces, and all compact Vaisman manifolds. 	We show there are computable invariants of a real homotopy type which in many cases prohibit it from containing any complex manifold satisfying such $dd^c$-type conditions in low degrees. This gives rise to numerous examples of almost complex manifolds which cannot be homotopy equivalent to any of these complex manifolds. 
\end{abstract}

\maketitle

%\allowdisplaybreaks
%\setcounter{tocdepth}{2}
%\setcounter{tocdepth}{1}
%\tableofcontents

\section{Introduction}

Well-formulated algebraic conditions can reveal deep connections within geometry and topology. This is epitomized in the work of Deligne, Griffiths, Morgan, and Sullivan, \cite{DGMS}, which introduces
the $d d^c$-condition for complex manifolds. This seemingly simple algebraic condition is a versatile tool in the study of compact complex manifolds for at least the following reasons:  

\begin{enumerate}
	\item It admits characterizations of rather distinct nature (using elements, indecomposable bicomplexes, the Fr\"olicher spectral sequence and pure Hodge structures, numerical inequalities).
	\item It passes to other manifolds in many geometric situations, such as holomorphic domination, projective bundles, small deformations, blow-ups (along $dd^c$-centers), etc.
	\item It holds on a fairly large class of manifolds, in particular, on compact K\"ahler manifolds.
	\item It implies topological restrictions on the underlying manifold: odd Betti numbers are even, and formality holds, in the sense of rational homotopy theory.
\end{enumerate}

In this paper we present a generalization of the $dd^c$-condition, termed the $dd^c+3$-condition, for which we obtain full analogues  of $(1)$--$(3)$ above. In the last section, we broaden our scope and 
provide a general framework for studying the real homotopy type of 
complex manifolds. This yields topological obstructions to the existence of complex structures satisfying a  low-degree variant of the $dd^c+3$-condition.

\subsection*{Definition and equivalent characterizations}
Let $M$ be a compact complex manifold and $\cA:=\cA(M)$ its bicomplex of $\CC$-valued differential forms.
One knows from \cite{DGMS} that the  $dd^c$-condition has several equivalent formulations (c.f. Theorem \ref{thm: ddc-cond} below). One such statement is that the following diagram induces an isomorphism in cohomology:
\[\label{fig: dc-diagram}
\xymatrix{
 & \left( \Ke d^c, d \right) \ar[dl]_i \ar[dr]^\pi &  \\
 \left(\cA,d \right)& & \left(H_{d^c} , d=0  \right)
}
\]
We perform a pushout on this diagram to reveal a new long exact sequence, valid for all complex manifolds, which is used in one formulation of our first main result:

\begin{thmintro}[The $dd^c + 3$-condition] \label{thmintro:ddc+3}
Let $(M,J)$ be a compact complex manifold.  The following are equivalent:
\begin{enumerate}
\item If $x \in \cA$ with $x = dy$, and $x= d^c z$, then $x = dw$ with $w \in \Ke d^c$.

\item The bicomplex $(\cA, \del, \delb)$ decomposes as a direct sum of dots, squares and length $3$ zigzags, i.e.:
\[
\xymatrix{   \CC_{p,q},  \\ }\quad
\xymatrix{
	\CC_{p,q+1} \ar[r]^-\del & \CC_{p+1,q+1} \\
	\CC_{p,q} \ar[r]^-\del  \ar[u]^-\delb & \CC_{p+1,q} \ar[u]_-\delb,
} \quad 
\xymatrix{
	\CC_{p,q+1}& \\
	\CC_{p,q} \ar[r]^-\del  \ar[u]^-\delb & \CC_{p+1,q},
} \quad
\xymatrix{
	\CC_{p,q+1}  \ar[r]^-\del  &  \CC_{p+1,q+1} \\
	& \CC_{p+1,q} \ar[u]_-\delb 
}
\]

\item The Fr\"olicher spectral sequence degenerates at $E_1$, and the total purity defect is at most $1$.
\item
The connecting homomorphism in the long exact sequence 
\[
\xymatrix{
\cdots \ar[r] & H^k_{d} \oplus  H^k_{d^c}  \ar[r]^-{p - j}  & H^k \left( \cA / \Img  \, d^c \right) \ar[r]^-{\delta_k}  & H^{k+1} \left( \Ke d^c \right) \ar[r]^-{i+\pi}  & H^k_{d} \oplus  H^k_{d^c} \ar[r] & \cdots 
}
\]
is zero for all $k$.

\item The following numerical equality holds:
\[ 
\sum_k \dim H^k(\Ke d^c) + \dim H^k(\cA/ \Img \, d^c) = 2 \sum_k \dim H^k_{\dR}(M).
\]

\end{enumerate}
\end{thmintro}
Every property above has a more restrictive counterpart that characterizes the usual $dd^c$-condition. In view of $(2)$, we call a complex manifold satisfying these conditions a $dd^c+3$-manifold. To make condition $(3)$ above precise, we introduce a non-negative integer that measures the extent to which the pure Hodge condition fails, called the \emph{purity defect}. While it can be defined in terms of filtrations, as in Definition \ref{defn:puritydefect}, it is easily (and equivalently) understood
in terms of lengths of the odd zigzags appearing in any decomposition of $(\cA,\del,\delb)$ into indecomposable bicomplexes, Proposition  \ref{prop:puritydefect}.  The conditions $(2)$ and $(3)$ have a natural generalization to higher length odd zigzags (resp. higher purity defect) and some of the results in this paper will hold for the resulting more general classes of manifolds.

The spaces  $H^* ( \Ke d^c )$ and $H^* ( \cA / \Img  \, d^c)$  in condition $(4)$ are closely related to the well-studied  Bott-Chern and Aeppli groups, $H^*_{BC}(M)$ and $H^*_A(M)$, respectively. In fact, these are pairwise isomorphic if and only if 
$H^*_{\dR}(M)$ inherits a pure Hodge structure, Theorem \ref{thm:purity1}. This is deduced from a natural diagram,
 respecting Poincar\'e duality, that contains all these groups, the long exact sequence, and certain purity obstruction groups, see subsection \ref{ssec:BCA}.

The numerical characterization in condition $(5)$ follows from a more general set of inequalities
\[
h_{BC} + h_A \geq h_{\Ke d^c} + h_{\cA / \Img\,d^c} \geq h_\delb + h_\del \geq  2  \sum_k b_k,
\]
where $b_k = \dim H^k_{\dR}(M)$, and the suggestive notation $h_{BC} = \dim H_{BC}$, $h_\delb = \dim H_\delb$, etc., is used.
This includes the case considered in \cite{AnTo}. All possible equalities above are characterized in Proposition \ref{prop;numerics} in terms of various degeneration and purity conditions, including the $d d^c + 3$ condition.
\subsection*{Inheritance of the $dd^c+3$-condition}
Just as for the $dd^c$-condition, the validity of the $dd^c+3$-condition is quite robust under many geometric constructions. In fact:
\begin{thmintro} \label{thmintro:constrmthds} 
The $dd^c+3$-condition satisfies:
\begin{enumerate}
	\item\label{itmint: blowup ddc+3} A blow-up of a manifold $M$ along a smooth center $Z\subseteq M$ is $dd^c+3$ if and only if both $M$ and $Z$ are $dd^c+3$.
	\item\label{itmint: ddbar x ddc+3} A product is $dd^c+3$ if and only if one factor is a $dd^c+3$-manifold and one is a $dd^c$-manifold.
	\item\label{itmint: ddc+3 modification}The target of a holomorphic surjection $f:M\to N$ with $M$ a $dd^c+3$-manifold and $\dim M=\dim N$ is again a $dd^c+3$-manifold.
	\item\label{itmint: ddc+3 proj. bdls} Projectivized holomorphic vector bundles are $dd^c+3$-manifolds if and only if the base of the bundle is a $dd^c+3$-manifold.
	\item\label{itmint: small defs} Any sufficiently small deformation of a $dd^c+3$-manifold is again a $dd^c+3$-manifold.
\end{enumerate}

\end{thmintro}
All of these statements have exact analogues replacing $dd^c+3$ by $dd^c$. Note however that $dd^c+3$-manifolds behave like a module over $dd^c$-manifolds, rather than as a ring themselves. In fact, the purity defect behaves additively under products. 

By means of the weak factorization theorem, one can extract statements about bimeromorphic invariants from $(1)$. For example:

\begin{corintro}
The $dd^c+3$-condition is a bimeromorphism invariant of compact complex manifolds in complex dimension at most four.
\end{corintro}

The deformation property in $(5)$ above is a consequence of the following more general statement:

\begin{thmintro} \label{thmintro: defs}  The condition ``$E_1$-degeneration and purity defect at most $k$'' 
is preserved under small deformations of compact complex manifolds.
\end{thmintro}

In the $dd^c$-case, small deformations have exactly the same cohomological invariants (Hodge numbers, Bott-Chern numbers, etc). This is in general not true in the setting of Theorem \ref{thmintro: defs}. However, under a slight technical strengthening of the $dd^c+3$-condition on the central fibre, satisfied by compact surfaces and Vaisman manifolds, the $E_1$-isomorphism type of the bicomplex of forms is constant under small deformations, see Corollary \ref{cor: E1isotype of Vaisman constant under small defs}. In particular, under this condition the Hodge and Bott-Chern numbers of nearby fibres agree with that of the central one.

\subsection*{Vaisman manifolds and the $dd^c+3$-condition}
Beyond $d d^c$-manifolds, there is an abundance of complex manifolds satisfying the $d d^c+3$-condition of Theorem \ref{thm:ddc+3},  including all compact complex surfaces (Corollary \ref{cor:surfaces}),
certain twistor spaces (Proposition \ref{prop:twddc+3}), and many nilmanifolds.

Our main example, however, are compact Vaisman manifolds \cite{Vais}, \cite{VaisGH}. Recall that a complex manifold is called Vaisman if it carries a Hermitian metric such that the fundamental form satisfies $d \omega =  \theta \wedge \omega$, with $\theta$ parallel. These form a large class of manifolds. For instance, given any projective manifold $M$ embedded as the zero section into a negative line bundle $L$, the quotient of $L\setminus M$ by the cyclic group generated by any complex number $\lambda\in\CC^\times\setminus S^1$ acting by translation in the fibres, carries a Vaisman metric \cite{VaisLCKGCK}. This construction generalizes the familiar examples of the Hopf manifolds.

\begin{thmintro}\label{thmintro: Vaisman}
	Compact Vaisman manifolds satisfy the $dd^c+3$-condition.
\end{thmintro}
It was previously known  for Vaisman manifolds that the Fr\"olicher spectral sequence degenerates at $E_1$, \cite[Thm 3.5]{Tsu}, so in view of Theorem \ref{thmintro:ddc+3}, the new contribution here is the control over the lack of purity in the cohomology. In fact, the theorem as stated is a consequence of a more general computation that precisely identifies which zigzags appear in which positions within the bicomplex of forms of a Vaisman manifold, Theorem \ref{thm: dc-str Vaisman}. 

This complete calculation allows one to draw many other conclusions, some of them yielding new and simple proofs of known results, like the fact that no Oeljeklaus-Toma manifold of type $(s,t)$ with $s\geq 2$ can be Vaisman \cite{KasVais}, or the very recent calculation of the Bott-Chern cohomology of a Vaisman manifold \cite{IO}.
Others are, to the best of our knowledge, new:

\begin{corintro}
	The middle cohomology of a compact Vaisman manifold of complex dimension $n+1$ carries a pure Hodge structure of weight $n+1$.
\end{corintro}

\begin{corintro}
	Every small deformation $V_t$ of a compact Vaisman manifold $V_0$ has the same $E_1$-isomorphism type, i.e. 
	for all $t$ sufficiently small:
	\begin{enumerate}
		\item The bicomplex $\cA(V_t)$ has the same zigzag multiplicities as $\cA(V_0)$,
		\item For any cohomological functor $H$ (e.g. $H_{BC},H_A,H_{\delb},..$),  $H(V_t)\cong H(V_0)$.
	\end{enumerate}
\end{corintro}

In Section \ref{Vanishing of higher multiplicative operations} we record some results on the vanishing of higher operations and Massey products on Vaisman manifolds, in analogy with the case of Sasakian manifolds, established in \cite{BFMT}. Together with the formality of $dd^c$-manifolds, this suggests a further study of the interplay between the $dd^c$-type conditions and the real homotopy type, which is carried out in the last section of the paper.

\subsection*{Homotopical restrictions imposed by $dd^c$-type conditions}
There are two ways to prove formality for $dd^c$-manifolds \cite{DGMS}: One consists 
in building a highly structured minimal model having a certain compatibility with the bigrading. The other, very quick one, consists in noting that the diagram $A\leftarrow \Ke d^c\rightarrow H(A)$ connects $A$ to its cohomology by quasi-isomorphisms. The second approach may at first seem to be very particular to the $dd^c$-setting. However, in the last section of the article, we turn it into a general technique to study the homotopy type of a complex manifold. Namely, we observe that the existence of a diagram $A\leftarrow B\rightarrow H(A)$ with certain extra properties (e.g. fixed ranks of the induced maps in cohomology) only depends on the homotopy type of $A$. On the other hand, for any complex manifold $M$ one obtains such a diagram for $A=\cA(M)$, and cohomological conditions on $M$ translate into conditions on the ranks of the induced maps in cohomology. Applying this kind of reasoning, we obtain:

\begin{thmintro} \label{thm;introhmtpyobst}
	Let $M$ be a compact manifold of dimension $2n$, with $j$-minimal model $\psi: \cM^j\to \cA(M)$ such that 
	\begin{enumerate}
		\item The map $H^{2n}(\psi)$ is surjective (i.e. the $j$-minimal model sees the fundamental class)
		\item The algebra $\langle H^{\leq j}(M)\rangle$ generated by cohomology classes in degree $\leq j$ has trivial intersection with $H^{j+1}(M)$ and $H^{2n}(M)$.
		\end{enumerate}
	If there is a complex manifold $N$ in the homotopy type of $M$ such that 
	\begin{enumerate}[resume]
		\item the natural map $\Ke d^c\to \cA(N)$ induces isomorphisms in $H^{s}$ for $s \leq j$, and
		\item the natural map $\Ke d^c\to \cA(N)\oplus H(N)$ induces an injection in $H^{j+1}$,
	\end{enumerate}
then $n=0$.
\end{thmintro}
The last condition is the $dd^c+3$-condition in degree $j+1$. The combination of the last two conditions can be recast in terms of which indecomposable bicomplexes can occur in $\cA(M)$, and also in terms of classical invariants like Hodge numbers and Hodge filtrations. The result as stated above is a less general (and less technical) version of the result in the main body of the text, which gives a topological lower bound on the complexity of the bicomplex of complex structures satisfying these $dd^c$-type conditions in low degrees, Theorem \ref{thm;TandAimpliesInequ}. That inequality is combined with a complex-analytic refinement of Poincar\'e duality, allowing one to relax the top-degree conditions above, Corollary \ref{cor;j<kLEQ2n}. 

Applying the Theorem with $j=1$, one obtains
\begin{corintro}
	The filiform nilmanifolds $G/\Gamma$, associated with the cdga of left invariant forms given by $\eta^1,...,\eta^{2n}$ s.t. $d\eta^k=\eta^1\wedge\eta^{k-1}$ cannot support a complex structure which satisfies the $dd^c+3$-condition with pure $H^1$.
\end{corintro}

It is known that the filiform nilmanifolds cannot admit left-invariant complex structures, and it is unknown whether they admit any complex structures at all (for $n\geq 3$). We stress that the same conclusion holds for any manifold rationally homotopy equivalent to a filiform nilmanifold, and for connected sums with any $1$-connected manifold. Further, since the conditions are only in very low degree, the result rules out many complex structures, and including many which are not $dd^c+3$. 
On the other hand, we give examples of $6$-dimensional manifolds which are complex, but are never $dd^c+3$ with pure $H^1$. We also give many non-nilmanifold examples, in particular rationally highly connected ones.\\

In the almost half-century since its appearance, the pioneering work of Deligne, Griffiths, Morgan, and Sullivan, has inspired a great number of applications related to K\"ahler geometry and rational homotopy theory. 
As we hope to demonstrate in this article, a return to these ideas sheds further light on complex geometry and its interaction with homotopy theory, far beyond the K\"ahler realm.\\

\noindent \textbf{Acknowledgements:} Large parts of this paper were written during a two-months stay of J.S. at the Graduate Center of the City University of New York. J.S. would like to thank Dennis Sullivan and the City University of New York for the invitation, financial support, and hospitality. We also thank Chi Li for a talk at the CUNY Graduate Center, which inspired the results in section \ref{ssec:stabilitydef}. The second author acknowledges support provided by a PSC-CUNY Award, jointly funded by The
Professional Staff Congress and The City University of New York.

\tableofcontents

\section{Preliminaries}
We recall some  definitions  and results that will be used below.

A \textit{bicomplex} (or  \textit{double complex}) is a bigraded $\CC$-vector space, $A=\bigoplus_{p,q \in \Z}A^{p,q}$, together with endomorphisms $\del$ and $\delb$,  of bidegrees $(1,0)$ and $(0,1)$, respectively, such that $d=\del+\delb$  satisfies $d^2=0$. Most of our bicomplexes will have a real structure, i.e. a complex anti-linear involution  $\sigma:A\to A$ such that $\sigma (A^{p,q})=A^{q,p}$ and $\sigma d \sigma=d$, hence we use the suggestive overline notation, but in general, no real structure is stipulated. Unless explicitly stated otherwise, we will always deal with bounded bicomplexes, i.e. those satisfying $A^{p,q}=0$ for all but finitely many $p,q \in \Z$. 

Our principal example is the space $\cA=\cA(M)$ of complex-valued forms on a complex manifold $M$, which further carries the structure of a graded-commutative differential graded algebra (cdga) and a real structure because it is the complexification of the space of real forms.

For any bicomplex $A$, one can form the \textit{column and row cohomology}, known as the Dolbeault and conjugate Dolbeault cohomologies, defined  by $H_{\delb}=\frac{\Ke\delb}{\Imt\delb}$ and $H_{\del}=\frac{\Ke{\del}}{\Imt\del}$. The column and row filtrations $F^pA=\bigoplus_{r\geq p} A^{r,s}$ and $\bar F^q=\bigoplus_{s\geq q}A^{r,s}$ induce  spectral sequences converging from these to the total cohomology $H_d= \frac{\Ke d}{\Imt d}$. The total cohomology has an induced \emph{pure Hodge structure (of weight $k$ in degree $k$)} if the two induced filtrations 
\[ F^pH^k_d(A)=\{[a]\mid a\in F^p A^k\}
\quad \textrm{and} \quad 
\bar F^q H^k_d(A)=\{[a]\mid a\in \bar F^q A^k\}
\] 
on $H^k$ are $k$-opposed, i.e. satisfy 
\[
H^k_d(A)=\bigoplus_{p+q=k}F^pH^k_d(A)\cap \bar F^q H^k_d(A). 
\]
This is equivalent to the condition that, for all $k$,
\[
F^pH^k_d(A)\cap \bar F^{k+1-p} H^k_d(A) =H^k_d(A)
\]
for all $p$, c.f. \cite{DeHII}.

A bicomplex is called \textit{indecomposable} if it cannot be written as a direct sum of two nontrivial sub-bicomplexes. Every indecomposable subcomplex is isomorphic to either a square, or a zigzag. The structure of these are recalled and indicated in diagrams below, when first needed in the proof of Theorem \ref{thm:ddc+3}. The length of a zigzag is its dimension as a vector space (i.e. the number of nonzero corners). Zigzags of length $1$ or $2$ will be called `dots' and `lines', respectively. Zigzags of length three with outgoing arrows will be called `L's', zigzags of length $3$ with incoming arrows will be called `reverse L's'. Any bicomplex can be written as a direct sum of indecomposable subcomplexes 
\[
A=\bigoplus_I I^{\oplus\mult_I(A)},
\] where $I$ runs over all squares and zigzags, and the multiplicity $\mult_I(A)$ of every isomorphism type of indecomposable bicomplex is the same in any such decomposition, c.f. \cite{KhQi}, \cite{StStrDbl}.

A map $f:A\to B$ of bicomplexes is called an \textit{$E_1$-isomorphism}, or bigraded weak equivalence, if it induces an isomorphism in both row and column cohomology. If both $A$ and $B$ have real structures, $\sigma_A$ and $\sigma_B$, and 
$\sigma_B f=f\sigma_A$, then $f$ is an $E_1$-isomorphism if and only if it induces an isomorphism in Dolbeault cohomology. This is the case for example for $A=\cA(M), B=\cA(N)$ for complex manifolds $M,N$ and $f=\varphi^*$ for some holomorphic map $\varphi: N \to M$. We write $A\simeq_1 B$ if there exists a chain of $E_1$-isomorphisms connecting $A$ and $B$. One has $\mult_Z(A)=\mult_Z(B)$ for any zigzag if and only if $A\simeq_1 B$, c.f. \cite{StStrDbl}. 

For any bicomplex, one can introduce the operator $d^c=\II^{-1} \, d \, \II$, where $\II$ acts on $A^{p,q}$ as multiplication by $i^{p-q}$. If $\left(\cA(M),d \right)$ is the differential forms of a complex manifold $(M,J)$, then $\II$ equals the extension of $J$ as an algebra automorphism, and $d^c = i (\delb - \del) $ is also a derivation. Let $H_{d^c}= \frac{\Ke d^c}{\Imt d^c}$ denote the cohomology of $\left(\cA(M), d^c\right)$, which is isomorphic to de Rham cohomology. 

Unless explicitly stated otherwise, we will assume all manifolds to be compact and connected.

\section{The $d d^c+3$-condition}

The results of this section are primarily algebraic and apply to any bounded bicomplex, while the main example of interest is the complex of $\CC$-valued smooth differential forms on a complex manifold. We'll use the abbreviated notation $\cA$ for either case, and highlight certain cases as appropriate.

\subsection{A long exact sequence} 

In this subsection we derive a new long exact sequence and observe that the vanishing of the connecting homomorphism in this sequence is a mild weakening of the so-called $d d^c$-condition. First we recall:

\begin{thm}\label{thm: ddc-cond} (The $d d^c$-condition, \cite{DGMS}, Theorem 5.7) For any bounded bicomplex $(\cA,\del, \delb)$, the following are equivalent:
\begin{enumerate}
\item For all $x \in \cA$, if $d^c x = 0$ and $x= d z$, then $x = d d^c w $ for some $w$.
\item The spectral sequences induced by the row and column filtrations both degenerate at $E_1$, and for each $k$ there is an induced pure Hodge structure of weight $k$ on $H^k(\cA)$.
\item The bicomplex $(\cA,\del, \delb)$ is a direct sum of
\begin{enumerate}
\item bicomplexes with only a single component, and $\del = \delb =0$,
\item bicomplexes which are a square of isomorphisms.
\end{enumerate}
\end{enumerate}
 \end{thm}
 
 Condition $(3)$ above can be equivalently stated as $\mult_I(\cA)=0$, unless $I$ is a dot or a square, and there are concise proofs now of the above theorem by checking the validity of statements $(1)$ and $(2)$ on every indecomposable complex, c.f. \cite[§2.3]{KhQi}, \cite[Cor. 7]{StStrDbl}.

Deligne, Griffiths, Morgan and Sullivan \cite{DGMS} show that all K\"ahler manifolds satisfy the $d d^c$-condition, and that
the $d d^c$-condition has strong implications for the topology of the underlying manifold.
To obtain the latter, one method is to show the following diagram
\[
\xymatrix{
 & \left( \Ke d^c, d \right) \ar[dl]_i \ar[dr]^\pi &  \\
 \left(\cA,d \right)& & \left(H_{d^c} , d=0  \right)
}
\]
is defined whenever $(M,J)$ satisfies the $d d^c$-condition, and that it induces an isomorphism in cohomology. From this it follows that $M$ is \emph{formal}, i.e. the differential graded algebra of differential forms on $M$ is connected by a chain of quasi-isomorphisms of differential graded algebras to its cohomology (equipped with zero differential). According to Sullivan's theory of rational homotopy, the rational homotopy groups are then a formal consequence of the cohomology groups, i.e. can be computed directly by a relatively simple procedure \cite{InfComp}.

Our first new observation is that the above diagram is well defined for all complex manifolds, even if the $d d^c$-condition does not hold. Let $(M,J)$ be an almost complex manifold and define $d^c=\II^{-1} \, d \, \II$ as before.
It is well known that $J$ is integrable if an only if $d$ and $d^c$ commute in the graded sense, i.e. $[d,d^c]=0$. This implies the existence of a diagram of cdga's as above, but what is not obvious is that $d=0$ on $H_{d^c}$. This follows from a more fundamental algebraic relation expressed in the proposition below.

\begin{prop} \label{Jdindent}
An almost complex structure $J$ is integrable if and only if:
\begin{align*}
[d, \JJ] &= d^c \\
[d^c , \JJ] &= - d ,
\end{align*}
where $\JJ$ denotes the extension of $J$ as a derivation. 
\end{prop}

The equations imply $[d,d^c]=0$. The two above equations are in fact equivalent, as can be seen by conjugating either by $\II$. Indeed, $\JJ$ and $\II$ commute, since $\JJ$ acts on $(p,q)$-forms by $i(p-q)$, which also gives the following beautiful formula:
\[
e^{\frac \pi 2 \JJ} = \II.
\]

\begin{proof}
On complex valued forms,   $d= \mub + \delb +\del + \mu$ with components of bidegrees $(-1,2)$, $(0,1)$, $(1,0)$ and $(2,-1)$, respectively.  In bidgree $(p,q)$, $\II^{-1} = (-1)^{p-q} \II$, so that conjugating an operator of bidegree $(r,s)$ by $\II$ acts by multiplication by $(-i)^{r-s}$. As operators on forms of bidegree $(p,q)$,
\[
[d, \JJ] 
=  3i \mub + i \delb - i \del - 3i \mu ,
\]
whereas
\[
\II^{-1} \, d \, \II 
=  -i \mub + i\delb - i \del + i \mu,
\]
so that
\[
[d,\JJ] - \II^{-1} \, d \, \II = 4 i( \mub - \mu).
\]
This vanishes if and only if $J$ is integrable.
\end{proof}
\begin{rmk}
Working with an arbitrary bigraded complex (not necessarily with multiplicative structure) one can define $\JJ$ as multiplication by $i(p-q)$ in bidegree $(p,q)$ and a similar argument shows an analogous characterization in the purely algebraic setting.
\end{rmk}
\begin{cor} \label{fundsquare}
For any complex manifold (resp. any bicomplex), there is a commutative diagram of complexes
\[
\xymatrix{
 & \left( \Ke d^c, d \right) \ar[dl]_i \ar[dr]^\pi &  \\
 \left(\cA,d \right) \ar[dr]_{p} & & \left(H_{d^c} , d = 0 \right) \ar[dl]^j \\
 & \left( \cA / \Img \, d^c , d \right) 
}
\]
This is both a pullback and a pushout in the category of complexes.
\end{cor}

\begin{proof}
Proposition \ref{Jdindent} implies that $d$ passes to $H_{d^c}$ with $d = 0 $, and that the two right maps  respect the differentials. The last statement is immediate to check.
\end{proof}

\begin{rmk}
The $dd^c$-condition holds in degree $k$ if and only if  
\[
H^k( \Img \, d^c,d) = 0 \quad \textrm{for all $k$}.
\]
 Thus the $d d^c$-condition holds if and only if one (and every) map in the diagram  
of Corollary \ref{fundsquare} is an isomorphism in cohomology. 
To see this, one inserts the complex $( \Img \, d^c,d)$ as the kernel or cokernel in all places, and passes to any of the long exact sequences in cohomology. 
\begin{comment}
{\small
\[
\xymatrix{
 & (\Img \, d^c,d) \ar[dr] & & \\
(\Img \, d^c,d)  \ar[dr] & & \left( \Ke \, d^c, d \right) \ar[dl]_i \ar[dr]^\pi &  \\
&\left(\cA,d \right) \ar[dr]_p \ar[dl]_{d^c}& & \left(H_{d^c} ,0  \right) \ar[dl]^j \\
 (\Img \, d^c,d)  & & ( \cA / \Img  \, d^c ,d) \ar[dl]_{d^c} & & \\
 & (\Img \, d^c,d) &  &  
}
\]
}
\end{comment}
\end{rmk}

A square of complexes
\[
\xy0;/r.25pc/:
(0,10)*+{A}="A";
(-10,0)*+{B}="B";
(10,0)*+{C}="C";
(0,-10)*+{D}="D";
(0,5)*+{\text{\rotatebox{135}{$\lrcorner$}}};
(0,-5)*+{\text{\rotatebox{-45}{$\lrcorner$}}};
{\ar_i"A";"B"};
{\ar^\pi"A";"C"};
{\ar^j"C";"D"};
{\ar_p"B";"D"};
\endxy
\]
is both a pullback and a pushout if and only if there is a short exact sequence
\[
\xymatrix{
0 \ar[r] & A=\Ke(p - j)  \ar[r]^-{i+\pi}  &  B \oplus  C  \ar[r]^-{p - j }  & \Cok(i+\pi)=D \ar[r] & 0.
}
\]
This gives the following:

\begin{thm} \label{LES}
For any complex manifold (resp. any bicomplex) there is a short exact sequence of complexes:
\[
\xymatrix{
0 \ar[r] &  \left( \Ke d^c, d \right) \ar[r]^-{i+\pi}  &  \left( \cA,d \right) \oplus  \left( H_{d^c},0 \right)  \ar[r]^-{p - j}  & \left( \cA / \Img \, d^c , d \right) \ar[r]  & 0,
}
\]
and therefore a long exact sequence in cohomology:
\[
\xymatrix{
\cdots \ar[r]^-{\delta_{k-1}} & H^k \left( \Ke d^c \right)  \ar[r]^-{i+\pi} & H^k_{d} \oplus  H^k_{d^c}  \ar[r]^-{p - j}  & H^k \left( \cA / \Img  \, d^c \right) \ar[r]^-{\delta_k}  & H^{k+1} \left( \Ke d^c \right) \ar[r]  &\cdots
}
\]
\end{thm}

\begin{rmk}
For $\cA=\cA(M)$ of a compact complex manifold $M$, the vector spaces appearing in this long exact sequence are all finite dimensional when the manifold is compact. One way to establish this is to relate them to the Bott-Chern and Aeppli cohomology groups, as is done in section \ref{ssec:BCA}.
\end{rmk}

The isomorphism type of this long exact sequence, and all things algebraically derived from it, are invariants of the biholomorphism type of complex manifolds. In particular, this holds for the rank of the connecting map $\delta$. We will later see that this rank is even a bimeromorphism invariant in complex dimension at most four, see Remark \ref{rmk;rank-delta-Bimer-Inv-Leq4}.

Finally, we relate the $d d^c$-condition to the long exact sequence from Theorem \ref{LES}.

\begin{lem}
If $\cA$ satisfies the  $d d^c$-condition, then the connecting homomorphism from Theorem \ref{LES} is zero in all degrees.
\end{lem}

\begin{proof}
Consider
\[
\xymatrix{
 H^k_{d} \oplus  H^k_{d^c}  \ar[r]^-{p - j}  & H^k \left( \cA / \Img  \, d^c \right) \ar[r]^-{\delta_k}  & H^{k+1} \left( \Ke d^c \right) \ar[r]^-{i+\pi} & H_d^{k+1} \oplus H_{d^c}^{k+1}
}
\]
The  $d d^c$-condition implies the first map, being the sum of two surjective maps, is surjective. Alternatively, the
$d d^c$-condition implies the last map, being the sum of two injective maps, is injective. By either argument, $\delta_k=0$.
\end{proof}

\subsection{Equivalent characterizations of $d d^c+3$} 

The vanishing of the connecting homomorphism $\delta$ does not quite imply the $d d^c$-condition. In fact, we have:

\begin{thm}[The $dd^c + 3$-condition]\label{thm:ddc+3}
For any bounded bicomplex $\cA$, the following are equivalent:
\begin{enumerate}
\item
The connecting homomorphism 
\[\delta_k: H^k \left( \cA/ \Img \, d^c \right)  \to H^{k+1} \left( \Ke d^c \right)
\]  
in the long exact sequence 
\[
\xymatrix{
\cdots \ar[r] & H^k_{d} \oplus  H^k_{d^c}  \ar[r]^-{p - j}  & H^k \left( \cA / \Img  \, d^c \right) \ar[r]^-{\delta_k}  & H^{k+1} \left( \Ke d^c \right) \ar[r]^-{i+\pi}  & H^k_{d} \oplus  H^k_{d^c} \ar[r] & \cdots 
}
\]
is zero for all $k$.
\item  For all $k \geq 0$, the diagram
\[
\xymatrix{
 &H^k ( \Ke \, d^c ) \ar[dl]_i \ar[dr]^\pi &  \\
 H^k_d \ar[dr]_{p} & & H^k_{d^c}  \ar[dl]^j \\
 & H^k ( \cA / \Img \, d^c ) 
}
\]
is both a pullback and a pushout in the category of vector spaces.
\item The following holds, for all $k\geq 0$: 

For all $x \in \cA^{k+1}$, if $x = dy$ and $x= d^c z$, then $x = dw$ with $w \in \Ke d^c$.
\item The following numerical equality holds:
\[ 
\sum_k \dim H^k(\Ke d^c) + \dim H^k(\cA/ \Img \, d^c) = 2 \sum_k b_k.
\]
(here, we make the additional assumption on $\cA$ that all involved quantities are finite).
\item The bicomplex $(\cA, \del, \delb)$ decomposes as a direct sum of dots, squares and length $3$ zigzags.

\item The Fr\"olicher (row- and column-) spectral sequences degenerate at $E_1$, and the total purity defect is equal to $1$.

\end{enumerate}
\end{thm}

The purity defect will be introduced in the subsection \ref{sec:purity}, where the equivalence of item $(6)$ will be proved.

\begin{rmk}
	If $\cA$ is equipped with a real structure $\sigma$ (e.g. if $\cA=\cA(M)$) then $d$ and $d^c$ are real operators and one may replace $\cA$ by the fixed points of $\sigma$ (the real forms) in $(1)$--$(4)$. In that setting, also, the two spectral sequences in $(6)$ are conjugate to each other, so it suffices to consider one.
\end{rmk}

\begin{proof} 
The sequence in $(1)$ has vanishing connecting homomorphism if and only if the long exact sequence splits into short exact sequences
\[
\xymatrix{
0 \ar[r]  & H^k ( \Ke d^c )  \ar[r]^-{i+\pi} & H^k_{d} \oplus  H^k_{d^c}  \ar[r]^-{p - j}  & H^k ( \cA / \Img  \, d^c ) \ar[r]   & 0 ,
}
\]
which holds if and only if diagram in $(2)$ is a pushout and pullback for all $k \geq 0$. 
To see the equivalence of the third condition, note that from the definitions we have
\[
H( \cA / \Img  \,d^c) = \frac{d^{-1}(\Img \,d^c)}{\Img \, d + \Img \,d^c} \quad
\textrm{and} \quad 
H(\Ke \, d^c) = \frac{\Ke \, d \cap \Ke \, d^c}{ d( \Ke \,d^c)}.
\]
Condition $(3)$ expresses that the connecting homomorphism $\delta$ has image zero.

The equivalence of condition $(4)$ to $(1)$ can be seen as follows.
For any compact complex manifold, the long exact sequence of Theorem \ref{LES} implies:
\[
\dim \left( H^k(\Ke d^c) / \Img \, \delta_{k-1} \right)  + \dim \left( \Ke \delta_k \right)= 2 b_k.
\]
Then the equality in $(4)$ holds if and only if $(1)$ holds.

Proving that condition $(5)$ is equivalent to $(1)$ will require several steps. First, the 
diagram 
\[
\delta_k: H^k ( \cA / \Img \,  d^c )  \to H^{k+1} ( \Ke d^c )
\]
can be understood as a functor on the category of bicomplexes over $\CC$, with values in the (linear) category of diagrams of two vector spaces and a linear map between them. Namely, for any bicomplex $(\mathcal{B}, \partial, \bar \partial)$ we let $d = \partial + \bar \partial$, and $d^c = i ( \bar \partial - \partial)$, and consider the diagram above. The map $\delta_k$ is readily seen to be induced by $d$. This functor is linear and it takes direct sums of bicomplexes to direct sums of vector spaces and maps between them, since $\Ke d^c$ and $\mathcal{B}/\Img \, d^c$ are compatible with direct sums. 

From \cite{KhQi},\cite{StStrDbl}, every (bounded) bicomplex $(\mathcal{B}, \partial, \bar \partial)$
decomposes as a direct sum of dots, squares, and zigzags, the definition of which we recall below.

 To complete the proof of the claim, it suffices to check that the map 
\[
\delta_j: H^j ( \cA / \Img \, d^c )  \to H^{j+1} ( \Ke d^c )
\]
is zero, for all $j$, on all on all bicomplexes which contain only dots, squares, and length $3$ zigzags, and that $\delta_j$ is non-zero for some $j$ on any bicomplex that contains an even length zigzag, or an odd zigzag of length $5$ or higher. We summarize these groups and the map $\delta_k$ in Figure \ref{fig:zzddc+3} below, which can be checked on a case-by-case basis, and explain the case of each row in the diagram. 

{\bf (Dot)} The case of a dot is a single vector space $\CC$, in some bi-degree $(p,q)$ with $k=p+q$, and zeroes elsewhere with vanishing differentials $\del$ and $\delb$. We compute that $\cA / \Img \, d^c  = \Ke d^c = \CC$ in degree $k$, with zero in all other degrees, so that $\delta_j : H^j ( \cA / \Img \, d^c ) \to H^{j+1}( \Ke d^c )$
is the zero map for all $j$. 

{\bf (Square)} A square is a bicomplex whose only non-zero entries are as follows, with maps that are isomorphisms:
\[
\xymatrix{
 \CC_{p,q+1} \ar[r]^-\del & \CC_{p+1,q+1} \\
 \CC_{p,q} \ar[r]^-\del  \ar[u]^-\delb & \CC_{p+1,q} \ar[u]_-\delb
}
\]
Let $k = p+q$.
Computing the total complexes $ (\cA / \Img \, d^c ,d)$ in total degrees $k,k+1,k+2$, we have $\CC \isomto \CC \to 0$. Similarly,  the complex $(\Ke d^c, d) $ in degrees $k,k+1,k+2$ is $0 \to \CC \isomto \CC$.
This shows the functors $H ( \cA / \Img \, d^c )$ and $H( \Ke d^c )$ vanish on squares, and clearly $\delta_j=0$ for all $j$.

{\bf (Odd length $3$ zigzag)} 
Consider an ``L'', i.e. a length $3$ zigzag with outgoing differentials that are isomorphisms:
\[
\xymatrix{
 \CC_{p,q+1}& \\
 \CC_{p,q} \ar[r]^-\del  \ar[u]^-\delb & \CC_{p+1,q}
}
\]
In this case the total complex $ (\cA / \Img \, d^c ,d)$ in degrees $k,k+1$ is $\CC \isomto \CC$, so the cohomology $ H (\cA / \Img \, d^c ,d)$ is zero in all degrees. Similarly, the complex  $(\Ke d^c, d) $ in total degrees $(k,k+1)$ is $0 \to \CC^2$ and the cohomology $H(\Ke d^c, d) $ is $(0,\CC^2)$ in degrees $k$ and $k+1$, respectively. Thus $\delta_j=0$ for all $j$.

Next we consider a ``reverse L'', i.e. a length $3$ zigzag with incoming differentials that are isomorphisms:
\[
\xymatrix{
 \CC_{p,q+1}  \ar[r]^-\del  &  \CC_{p+1,q+1} \\
 & \CC_{p+1,q} \ar[u]_-\delb 
}
\]
In this case the total complex $ (\cA / \Img \, d^c ,d)$ in total degrees $(k+1,k+2)$ is $\CC^2 \to 0$, so the cohomology $ H (\cA / \Img \, d^c ,d)$ is $(\CC^2,0)$ in total degrees $k+1$ and $k+2$, respectively. Similarly, the complex  $(\Ke d^c, d) $ in total degrees $(k+1,k+2)$ is $\CC\isomto \CC$ and the cohomology $H(\Ke d^c, d) $ is $(0,0)$ in degrees $k+1$ and $k+2$, respectively. Thus $\delta_j=0$ for all $j$.

{\bf (General odd-length zigzag)} For general odd-length zigzags we have two cases, outgoing and incoming:
\[
\xymatrix{
 \CC & & &  & \\
 \CC \ar[r]^-\del  \ar[u]^-\delb & \CC & & & \\
  & &  \CC \ar[r]^-\del    \ar@{}[ul]|-{\ddots} &\CC & \\
 &  & & \CC \ar[r]^-\del \ar[u]^-\delb& \CC
}
\quad 
\xymatrix{
 \CC \ar[r]^-\del & \CC & &  & \\
  & \CC \ar[r]^-\del  \ar[u]^-\delb & \CC& & \\
 & & & \CC \ar@{}[ul]|-{\ddots} \ar[r]^-\del & \CC \\
&  & & &  \CC \ar[u]^-\delb 
}
\]
Consider the first case, which is in total degrees $k$ and $k+1$, with vector spaces $\CC^m$ and $\CC^{m+1}$, respectively.
The total complex $ (\cA / \Img \, d^c ,d)$ is $\CC^m$ in degree $k$, and $\CC$ in degree $k+1$, since $d^c$ is injective on each copy of $\CC$ in degree $k$, but not onto. The differential $d: \CC^m \to \CC$ is onto, so the total cohomology of $H(\cA / \Img \, d^c ,d)$ is $\CC^{m-1}$, and $0$, in total degrees $k$ and $k+1$, respectively. 

Continuing with this odd-length outgoing case, the total complex $ (\Ke d^c ,d)$ is 
$0$ in degree $k$, and $\CC^{m+1}$ in degree $k+1$. So the total cohomology of $H(\Ke d^c ,d)$ is $0$ and $\CC^{m+1}$, in total degrees $k$ and $k+1$, respectively. Finally, the differential $\delta_k:H^k(\cA / \Img \, d^c ,d)\to H^{k+1}(\Ke d^c ,d)$ is the injection $\CC^{m-1} \to \CC^{m+1}$.

The case of odd-length incoming is computed similarly. The results are in Figure \ref{fig:zzddc+3}, and yields that $\delta_k :H^k(\cA / \Img \, d^c ,d)\to H^{k+1}(\Ke d^c ,d)$ is the surjection $\CC^{m+1} \twoheadrightarrow \CC^{m-1}$, is non-zero for $m > 1$.

{\bf (Even-length zigzags)} For general odd-length zigzags we again have two cases, where the top-leftmost space has an outgoing or incoming map:
\[
\xymatrix{
 \CC \ar[r]^-\del & \CC & &  & \\
  & \CC  \ar[u]^-\delb & &  &\\
 & & \CC \ar[r]^-\del  \ar@{}[ul]|-{\ddots}  & \CC & \\
&  & &  \CC \ar[u]^-\delb \ar[r]^-\del & \CC
}
\quad
\xymatrix{
 \CC & & &  & \\
 \CC \ar[r]^-\del  \ar[u]^-\delb & \CC & & & \\
  & &  \CC \ar[r]^-\del  \ar@{}[ul]|-{\ddots} &\CC & \\
 &  & & \CC \ar[u]^-\delb&
}
\]
Suppose each complex has $\CC^m$ in total degrees $k$ and $k+1$.
Here the two cases yield the same complexes in total degree. Namely, in either case, the complex $(\cA / \Img \, d^c ,d)$ is $\CC^m \to 0$  with the same cohomology, and the complex $(\Ke d^c ,d)$ is $0 \to \CC^m$, with same cohomology. The differential $\delta_k: H^k(\cA / \Img \, d^c ,d) \to H^{k+1}(\Ke d^c ,d)$ is the isomorphism $\CC^m \isomto \CC^m$, which is a non-zero for $m>0$, i.e. length at least $2$.

{\small
\begin{figure}

\begin{tabular}{ c| c | c | c | c | c | c}
zigzag type & length & $ \cA / \Img \, d^c$  & $\Ke d^c$  &  $H (\cA / \Img \, d^c )$ & $H(\Ke d^c )$ & $\textrm{rank} (\delta)$ \\  \hline 
Dot & $1$ & $\CC$ & $\CC$ & $\CC$ & $\CC$ & $0$ \\
Square & NA &  $\CC \to \CC \to 0$ & $0 \to \CC \to \CC$ & $(0,0,0)$  & $(0,0,0)$ & $0$ \\
$L$ &  $3$ & $\CC \to \CC$  & $0 \to \CC^2$ & $(0,0)$ & $(0,\CC^2)$ & $0$ \\
Rev. $L$ & $3$ & $\CC^2 \to 0$ & $\CC \to \CC$  & $(\CC^2 , 0) $& $(0,0)$  & $0$ \\
Odd Out. & $2m+1$ & $\CC^m \twoheadrightarrow \CC$ & $0 \to \CC^{m+1}$ & $(\CC^{m-1},0)$ & $(0,\CC^{m+1})$ & $m-1$ \\
Odd Inc. & $2m+1$ & $\CC^{m+1} \to 0$ &  $\CC \hookrightarrow \CC^{m}$ & $(\CC^{m+1},0)$ & $(0,\CC^{m-1})$ & $m-1$ \\
Even Out. & $2m$ & $\CC^m \to 0$ & $0 \to \CC^m$ & $(\CC^m, 0) $ &  $(0 ,\CC^m)$  & $m$ \\
Even Inc. & $2m$ & $\CC^m \to 0$ &  $0 \to \CC^m$ & $(\CC^m , 0)$ & $(0 , \CC^m)$ & $m$ 
\end{tabular}

 \caption{Zigzag contributions for the connecting map ${\delta_k: H^k(\cA / \Img \, d^c ,d) \to H^{k+1}(\Ke d^c ,d)}$}\label{fig:zzddc+3}
\end{figure}
 }
\end{proof}

\begin{defn} The equivalent conditions in Theorem \ref{thm:ddc+3} will be referred to as the \emph{$d d^c+3$-condition}. 
A complex manifold $M$ will be said to be $d d^c+3$ if the bicomplex of differential forms on $M$
satisfies the  $d d^c+3$-condition.
\end{defn}

\begin{ex}
	For $M=S^1\times S^3$ with the complex structure of a Hopf manifold, the Fr\"olicher spectral sequence degenerates and $b_0=b_1=1$, $b_2=0$. Hence one has (see e.g. \cite[Ch. 4]{StStrDbl})
	\[
	\cA(M)\simeq_1 \img{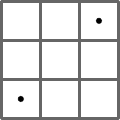}\oplus\img{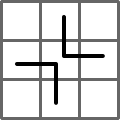}.
	\]
	Hence, it is $dd^c+3$. We will  generalize this below in two ways  (to all complex surfaces and to all Vaisman manifolds).
\end{ex}
The $d d^c + 3$-condition fails in general in complex dimension greater than $2$, as the following example shows.

\begin{ex} \label{S3S3} 
For $M=S^3 \times S^3$ with the Calabi-Eckmann complex structure, $h^{0,1}(M)\neq 0$ by \cite{Bor}, so the Fr\"olicher spectral sequence does not degenerate. Thus, $M$ is not $dd^c+3$. One can analyze this failure more precisely. In fact, one may extract from the calculations in \cite[§3.3]{ATBCf} that
\[
\cA(M)\simeq_1 \img{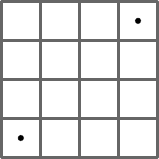}\oplus \img{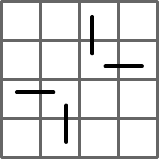}\oplus\img{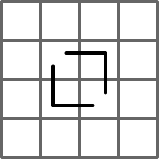}.
\]
By inspecting Figure \ref{fig:zzddc+3} one sees that the connecting homomorphism
\[
\delta_k: H^k \left( \cA / \Img \, d^c \right)  \to H^{k+1} \left( \Ke d^c \right)
\] 
is an isomorphism for $k=1,4$, and the source and target are non-zero (in fact 2-dimensional) in this case.
\end{ex}

We end this section with a remark concerning the other potential extremity of the connecting homomorphism.

\begin{rmk}
The connecting homomorphism $\delta$ is an isomorphism (away from the top and bottom degrees) if and only if $M$ is a rational homology sphere. Note that beyond the standard sphere $S^6$, there are numerous rational homology $6$-spheres that are almost complex. In fact, as shown by (\cite{AM}, p.5), performing surgery on the first factor of $S^1 \times N$ preserves the condition of being spin$^c$, which in dimension $6$ is equivalent to being almost complex. Applying this construction to those $5$-dimensional lens spaces which are spin$^c$ yields infinitely many topologically distinct examples. It is not known if any posses a complex structure.
\end{rmk}
 
\subsection{Purity defect} \label{sec:purity}

We relate the $d d^c +3$-condition to a modest failure of the pure Hodge condition. To do this, we introduce a definition that measures the extent to which a complex manifold fails to have a pure Hodge structure, namely a non-negative integer, called the \emph{purity defect}, defined in terms of the Hodge filtration. Proposition \ref{prop:puritydefect} shows that this number simply measures the longest odd length zigzag in the bicomplex $(\cA,\del, \delb)$. It will follow that a manifold satisfies the $d d^c +3$-condition if and only if there is first page degeneration and purity defect at most $1$, Corollary \ref{cor:ddc+3iffE1andpd1}. Again, the definition and the last mentioned Corollary work just as well for any (bounded) bicomplex, but for ease of language and because we show some specifically geometric results, we work with the complex of differential forms in this subsection.

Recall that for any complex manifold $(M,J)$  the Hodge-filtration,
\[
F^p \cA^k(M) := \bigoplus_{j\geq p} \cA^{j,k-j} (M),
\]
induces a filtration on the de Rham cohomology, via
\[
 F^p  H^k_{d R}(M;\CC)  := \Img\left(F^p \cA^k(M) \cap \Ke d\longrightarrow H_{\dR}^k(M;\CC)\right),
\]
as the space of de Rham classes that are representable by forms with holomorphic bidegree greater than or equal to $p$. We let
$\bar F$ denote the conjugate filtration and we say that $H^k_{d R}(M)$ inherits a pure Hodge structure  (of weight $k$ in degree $k$) if
\[
H^k_{d R}(M ;\CC) = \bigoplus_{p+q=k}  F^p  H^k_{d R}(M;\CC) \cap \bar F^q  H^k_{d R}(M;\CC).
\]

\begin{defn}\label{defn:puritydefect} Let $(M,J)$ be a complex manifold.
\begin{enumerate}
	\item The \emph{total filtration} is the descending filtration defined by 
	\[
	F_{tot}^rH_{\dR}^k(M):=\sum_{p+q=r}F^pH^k_{\dR}(M;\CC)\cap \bar{F}^qH^k_{\dR}(M;\CC).
	\]
	\item The \emph{purity defect in degree $k$} is defined to be:
	\[
	\pdef_k(M):=\max \bigg\{ |d| \, \bigg| \, d\in\Z \,\, \text{and}  \,\, \operatorname{gr}_{F_{tot}}^{k+d}\! H_\dR^k(M;\CC)\neq 0 \bigg\},
	\]
	where we understand the maximum to be $0$ if $H_{\dR}^k(M;\CC)=0$.
	\item The \emph{(total) purity defect} is the nonnegative integer
	\[
	\pdef(M):=\max_k \pdef_k(M).
	\]
\end{enumerate}
\end{defn}

Note that $F_{tot}^r=0$ for $r$ sufficiently large. In particular, 
\[
F^p  H^k_{d R}(M;\CC) \cap \bar F^q  H^k_{d R}(M;\CC) =0
\]
 whenever $p+q-k$ is greater than the purity defect in degree $k$. 
This observation admits a sort of converse. As a consequence of Serre duality, on any connected compact complex manifold, one has nondegenerate pairings (cf. \cite{SteLR})
\[
\gr^{k+d}_{F_{tot}}H_\dR^k(M;\CC)\times\gr_{F_{tot}}^{2n-k-d} H_\dR^{2n-k}(M;\CC)\to\CC.
\]
 Thus, we have:
\begin{lem}
	On a connected compact complex manifold $M$, the purity defect $\pdef(M)$ is the absolute value of the maximal number $p+q-k$ such that 
	\[ 
	F^pH_\dR^k(M;\CC)\cap \bar F^qH_\dR^k(M;\CC)\neq 0.
	\]
\end{lem}

\begin{ex}\label{expl: purdef} We list some low degree examples and a bound of $\pdef(M)$ in terms of the dimension of $M$. 
Let $H^k = H^k_{d R}(M;\CC)$.
\begin{enumerate}
\item The purity defect is zero in degree $k$ if and only if 
\[
F^p  H^k \cap \bar F^q  H^k = 0 \quad \quad \textrm{for $p+q >k$},
\]
and 
\[
\sum_{p+q=k} F^p H^k\cap\bar F^p H^k=H^k.
\]
Equivalently, this holds if and only if there is a pure Hodge structure  of weight $k$ in degree $k$.  This in particular implies the bicomplex $(\cA, \del ,\delb)$ has no odd zigzags of positive length, c.f. \cite{DGMS} \cite{StStrDbl}. 

\item The purity defect is at most $1$ in degree $k$ if and only if
\[
F^p  H^k \cap \bar F^q  H^k = 0 \quad \quad \textrm{for $p+q  > k+ 1$},
\]
and 
\[
\sum_{p+q=k-1} F^p H^k\cap\bar F^p H^k=H^k.
\]
Thus, if a nonzero class $[\omega^{p,q}]=[\omega^{r,s}]$ has two pure representatives of types $(p,q)$ and $(r,s)$, then $|p-r|\leq 1$, and additionally, every class $\mathfrak{c}\in H^k$ can be written as a sum of classes $\mathfrak{c}=\sum \mathfrak{c}_i$, where each $\mathfrak{c}_i=[\omega^{p_i,q_i}+\omega^{p_i+1,q_i-1}]$ is representable by a closed forms with at most two neighboring components.
In particular, for any $n$-dimensional complex manifold $M$, one has $\pdef_1(M)=\pdef_{2n-1}\leq 1$. 
As we see in the proposition below, a purity defect of at most $1$ implies all odd zigzags are length at most $3$.

\item On any compact complex manifold of complex dimension $n$, Serre duality and bidegree reasons imply
\[
\pdef_k(M)=\pdef_{2n-k}(M)\leq k.
\] 
In particular, $\pdef(M)\leq n$. The inequality will be improved in Corollary \ref{cor: pdef leq n-1} below to $\pdef(M)\leq n-1$.
\end{enumerate}
\end{ex}

\begin{prop}  \label{prop:puritydefect}
A complex manifold $(M,J)$ has $\pdef(M)\leq \ell$ if and only if there are no zigzags of odd length greater than $2\ell +1$, in any decomposition of $\cA(M)$ into indecomposables.
\end{prop}

This follows  because the multiplicity of odd-length zigzags with length $2|p+q-k|+1$ are measured by the refined Betti-numbers
\[
b_k^{p,q}(M)= \dim \frac{F^pH^k\cap \bar F^qH^k}{F^{p+1}H^k\cap \bar F^qH^k+F^pH^k\cap \bar F^{q+1}H^k},
\]
where $H^k:=H^k_\dR(M;\CC)$, see \cite{StStrDbl}. We recall the simple idea behind this formula, which also explains Proposition \ref{prop:puritydefect}.  Consider an odd zigzag of the form
\[
\xymatrix{
 \CC & & &  & \\
 \CC \ar[r]^-\del  \ar[u]^-\delb & \CC & & & \\
  & &  \CC \ar[r]^-\del    \ar@{}[ul]|-{\ddots} &\CC & \\
 &  & & \CC \ar[r]^-\del \ar[u]^-\delb& \CC
}
\]
The total cohomology is one-dimensional and represented both by a generator for the top left corner, and by the cohomologous generator for the bottom right corner. So, the longer such a zigzag is, the greater the possible value $r=p+q$ such that $F^pH^k_{\dR}(M;\CC)\cap \bar{F}^qH^k_{\dR}(M;\CC) \neq 0$, and the greater is the purity defect. A similar calculation can be done for odd zigzags with incoming arrows.

\begin{cor}\label{cor: pdef leq n-1}
	For any compact complex manifold of complex dimension $n$, the inequality $\pdef(M)\leq n-1$ holds.
\end{cor}
\begin{proof}
	By Proposition \ref{prop:puritydefect},  $\pdef(M)=n$ implies that there is a zigzag of length $2n+1$, which, for space reasons, would have to have a nonzero component in bidegree $(n,0)$. However, by an application of Stokes' theorem one may see that the only indecomposable complexes with nonzero components in degree $(n,0)$ are dots and squares (c.f. \cite[Ch. 4]{StStrDbl}).
\end{proof}

The following corollary completes the proof of Theorem \ref{thm:ddc+3}, showing $(5)$ is equivalent to $(6)$.

\begin{cor} \label{cor:ddc+3iffE1andpd1}
The $d d^c + 3$-condition holds if and only if the Fr\"olicher spectral sequence degenerates at $E_1$ and the purity defect is at most $1$.
\end{cor}

\begin{proof}
Degeneration at $E_1$ occurs if and only if there are no even zigzags, and purity defect at most $1$ occurs if and only if there are no odd zigzags of length greater than $3$. The two together are equivalent to the condition that a bicomplex $(\cA,\del,\delb)$ decomposes into a direct sum of dots, squares, and length three zigzags.
\end{proof}

 \begin{cor}
	The  $dd^c$-condition holds if and only if the $dd^c + 3$ condition holds and $H_{d R}(M;\CC)$ has a pure Hodge structure in all degrees, i.e. purity defect is zero.
\end{cor}

\begin{proof} 
	This follows from Theorem \ref{thm: ddc-cond} and Example \ref{expl: purdef} (1) (or Proposition \ref{prop:puritydefect}). 
\end{proof}

\begin{rmk}
	The last two Corollaries suggest a natural generalization of the $dd^c+3$ condition: Consider those manifolds with degenerate Fr\"olicher spectral sequence and a fixed bound on the purity defect (i.e. no even zigzags and bounded length of odd zigzags). This type of condition will also naturally re-appear in later sections.
\end{rmk}

\begin{prop}\label{prop: purdef of prod}
	For any compact complex manifolds, $M$ and $N$,
	\[
	\pdef(M\times N)=\pdef(M)+\pdef(N).
	\]
\end{prop}
 \begin{proof}
  Let $Z_M$ be an odd zigzag of maximal length $2m+1$ in $\cA(M)$ and $Z_N$ an odd zigzag of maximal length $2n+1$ in $\cA(N)$. By Serre duality, we may assume they both have incoming outermost arrows. Now, there is an $E_1$-isomorphism $\cA(M\times N)\simeq_1 \cA(N)\otimes \cA(M)$ and so $Z_M\otimes Z_N$ is a direct summand in $\cA(M\times N)$. But $Z_M\otimes Z_N\simeq_1 Z_{M\times N}$ where $Z_{M\times N}$ is an odd zigzag of length $2(n+m)+1$. Thus $\pdef(M)+\pdef(N)\leq \pdef(M\times N)$. The argument also works the other way, since a tensor product of an even length zigzag with any other bicomplex does not contain odd zigzags (\cite[Ch. 3]{StStrDbl}).
 \end{proof}
 We conclude with the following curious observation:
 \begin{prop}\label{prop: mult Hodge}
 	If $\pdef(M)\leq 1$, the cohomology algebra $H(\Ke d^c)$ carries a multiplicative Hodge structure, i.e., the cohomology groups $H^k(\Ke \, d^c)$ admit a real Hodge structure of weight $k$ such that the cup product restricts to maps $H^{p,q}(\Ke d^c)\otimes H^{p',q'}(\Ke d^c)\longrightarrow H^{p+p',q+q'}(\Ke d^c)$. (c.f. \cite{Voi08}).  
	 \end{prop}
 
 \begin{proof}
 	First note that the Hodge filtrations on $\Ke d^c$ (induced by row and column filtration on $\cA(M)$) are compatible with the wedge product, since they are on $\cA(M)$. In particular, the cup product on cohomology respects these filtrations, i.e. $F^pH^r(\Ke d^c)\cup F^qH^s(\Ke d^c)\subseteq F^{p+q}H^{r+s}(\Ke d^c)$ and similarly for $\bar F$. Now we argue via indecomposable bicomplexes: For $I$ any square, or reverse $L$, $\Ke d^c (I)$ is contractible.  For $I$ any even length zigzag, dot or $L$, the bicomplex structure on $I$ induces a bicomplex structure on $\Ke d^c(I)$, which is then a direct sum of dots. Thus, in all cases that have $\pdef(I)\leq 1$, the Hodge filtrations induce a pure Hodge structure on $H^k(\Ke d^c)$. 
 \end{proof}
 
 As noted in \cite{Voi08}, the existence of a multiplicative Hodge structure on an algebra $H$ imposes further conditions beyond $b_{2k+1}$ are even. For example, the image of the cup product maps $\Img(\cup:H^k\otimes H^{l}\to H^{l+k})$ are sub-Hodge structures, and thus have even rank whenever $l+k$ is odd.
 
 \subsection{Relation to Bott-Chern and Aeppli cohomologies} \label{ssec:BCA}

 Recall the Bott-Chern and Aeppli cohomologies are defined as follows:
 \[
 H^k_{BC}(\cA) = \frac{\Ke d \cap \Ke d^c}{ \Img \,d d^c} \cap \cA^k
 \quad \quad 
 H^k_A (\cA) = \frac{\Ke d d^c}{\Img \,d + \Img \,d^c} \cap \cA^k.
 \]
 It is well known that these are finite dimensional, that $H_{BC}$ is a bi-graded algebra, that $H_A$ is a bi-graded module over $H_{BC}$, and that for any choice of metric, $\star: H_{BC} \to H_A$ is an isomorphism on complementary degrees. Moreover, there is a well defined natural transformation,
 \[
 d: H^*_A  \to H^{*+1}_{BC}.
 \]

 \begin{prop} \label{comparegroups}
 	For any complex manifold (resp. any bicomplex) there is a natural surjection 
 	\[
 	\phi: H^*_{BC} (\cA)  \twoheadrightarrow H^*(\Ke d^c) \quad \textrm{with} \quad \Ke(\phi) = \frac{d(\Ke d^c)}{\Img \, d d^c},
 	\]
 	and injection
 	\[
 	\psi: H^* ( \cA / \Img  \, d^c ) \hookrightarrow H^*_A (\cA)  \quad \textrm{with} \quad \Cok(\psi) = \frac{\Ke d d^c}{d^{-1} \Img \,d^c},
 	\]
 	induced by the identity map.
 	
 	In particular, for any compact complex manifold, the groups  $H^k ( \Ke d^c )$  and $H^k ( \cA / \Img  d^c )$, as well as $ \Ke(\phi) $ and 
 	$ \Cok(\psi)$, are finite dimensional for all $k$. 
 \end{prop}
 
 \begin{proof}
 	This follows since
 	\[
 	H^*(\Ke d^c) = \frac{\Ke d \cap \Ke d^c}{ d( \Ke \,d^c)}
 	\quad \quad
 	H^*(\cA/ \Img \,d^c) = \frac{d^{-1} \Img \, d^c}{\Img \,d + \Img \, d^c}.
 	\]
 \end{proof}

 \begin{defn} \label{defn:puritygroups}
 	For any bicomplex $\cA$ define the  \emph{obstruction to purity groups} 
 	\[
 	H^k_\urcorner (\cA) =  \frac{d(\Ke d^c)}{\Img \, d d^c} \, \cap \, \cA^k \quad \quad
 	\quad \quad
 	H^k_\llcorner (\cA) = \frac{\Ke d d^c}{d^{-1} \Img \, d^c} \, \cap \,\cA^k \quad \quad
 	\]
 \end{defn}

 In summary, we have a diagram
 \[
 \xymatrix{
 	& 0 \ar[d] & 0  & \\
 	\cdots H^k_{d} \oplus  H^k_{d^c} \ar[r] & H^k ( \cA / \Img  \, d^c ) \ar@{^{(}->}_{\psi}[d]  \ar[r]^{\delta_k} & H^{k+1}(\Ke \, d^c)  \ar[u]  \ar[r] & H^{k+1}_{d} \oplus  H^{k+1}_{d^c} \cdots \\
 	& H^k_A  (\cA) \ar[r]^d \ar@{>>}[d] & H^{k+1}_{BC}(\cA)  \ar@{>>}_{\phi}[u] & \\
 	& H^k_\llcorner  (\cA)  \ar[d] \ar[r]^{d=0}& H^{k+1}_\urcorner  (\cA)   \ar@{^{(}->}[u] & \\
 	& 0 & \ar[u] 0 & 
 }
 \]
 
 \begin{prop} \label{Pduality}
 	On a compact connected $n$-dimensional complex manifold $M$, the integration pairing $\omega\mapsto \int_M \omega\wedge -$ induces a duality between the vertical short exact sequences,  in that
 	\[
 	H^k (\Ke d^c) \cong (H^{2n-k}  (\cA/ \Img d^c))^\vee \quad \textrm{and} \quad 
 	H^k_\urcorner (M)\cong (H^{2n-k}_\llcorner (M))^\vee,
 \textrm{ for all }k.	\] 
 \end{prop}
 \begin{proof}
  This follows from a general duality statement for cohomological functors. We do the first case in detail. The linear functor sending a bicomplex $(A,\del,\delb)$ to $H^k(\Ke d^c(A))$ sends squares to zero and commutes with arbitrary direct sums, and so defines a cohomological functor. Denote by $D\cA$ the dual bicomplex, as a bigraded vector space given by $(D\cA)^{p,q}=(A^{n-p,n-q})^\vee$ with differential $d_{D\cA}=(\varphi\mapsto(-1)^{|\varphi|-1}\varphi\circ d)$. Since $M$ is compact and oriented, the duality map $\cA\to D\cA$ given by $\omega\mapsto \int_M \omega\wedge -$ induces an isomorphism on all cohomological functors \cite[Cor. 20]{StStrDbl}, so in particular
 \[H^k(\Ke d^c(\cA))\cong H^k(\Ke d^c(D\cA)).
 \]
 Now it is a linear algebra calculation to show that 
 \[
 H^k(\Ke d^c(D\cA))\cong (H^{2n-k}(\cA/\Img \, d^c))^\vee.
 \]
 In fact, since over fields cohomology commutes with duals, this follows from the identification \[[\Ke d^c(D\cA)]^k=\{\varphi\in [\cA^{2n-k}]^\vee \mid \varphi\circ d^c=0\}\cong ([\cA/\Img \, d^c]^{2n-k})^\vee.
\]
 \end{proof}

 The following result will characterize the case in which $\phi$ and $\psi$ are isomorphisms, in terms of the existence of a pure Hodge structure.

 \begin{thm} \label{thm:purity1} 
 	Let $\cA$ be a bicomplex. The following are equivalent:
 	\begin{enumerate}
	\item $\pdef(\cA)=0$.
 		\item The de Rham cohomology $H^k_{d R}(\cA)$ inherits a pure Hodge structure in all degrees $k$.
 		\item The purity obstruction groups vanish for all $k$,
 		\[
 		H^k_\urcorner (\cA)  =0  \quad \textrm{and} \quad H^k_\llcorner (\cA)  = 0.
 		\]
 		\item The natural maps induced by the identity,
\[
\phi:  H^k_{BC} (\cA)  \to H^k(\Ke d^c)  \quad \textrm{and} \quad
\psi:  H^k \left( \cA / \Img  \,d^c \right) \to  H^k_A (\cA) ,
\]
are isomorphisms for all $k$.
 	\end{enumerate}
 \end{thm}
 
 \begin{proof} It remains only to show $(3)$ is equivalent to $(2)$.
 	Let us focus on the groups $H_\urcorner$ first, where a similarly calculation can be done for $H_\llcorner$, or, if $\cA=\cA(M)$ for a compact manifold, one may appeal to Serre Duality as in Proposition \ref{Pduality}. 
 	
 	Recall that $H_{\dR}(\cA) $ has a pure Hodge structure if and only if there are only squares, even zigzags, and dots, with no odd zigzags of length greater than one, \cite{StStrDbl}. The proof then proceeds by computing either $H_\urcorner(\cA) $ or $\phi$ on every type of indecomposable complex. The results are summarized in Figure \ref{fig:zigzagkerdctoHBC}.
 \end{proof}

 	{\small
 		\begin{figure}
 					\begin{tabular}{ c| c | c | c | c | c }
 					zigzag type & length & $H_\urcorner (\cA)$ & $H_{BC}(\cA) $    &  $H(\Ke d^c )$ & $\phi$ \\  \hline 	Dot & $1$ & $0$ & $\CC$ &  $\CC$ & Iso \\
 					Square & NA &  $0$ &  $(0,0,0)$  & $(0,0,0)$ & Iso \\
 					Odd Out. & $2m+1$ & $0$ & $(0,\CC^{m+1})$ & $(0,\CC^{m+1})$ & Iso \\
 					Odd Inc. & $2m+1$ & $\CC$ & $(0,\CC^{m})$ & $(0,\CC^{m-1})$ & Surj \\
 					Even Out. & $2m$ & $0$ &  $(0 ,\CC^m)$   &  $(0 ,\CC^m)$  & Iso \\
 					Even Inc. & $2m$ & $0$ &  $(0 , \CC^m)$ & $(0 , \CC^m)$ & Iso 
 				\end{tabular}
 		\caption{Zigzag contributions for the short exact sequence ${0\to H_\urcorner (\cA)  \to  H_{BC}(\cA)  \to H(\Ke \, d^c)\to 0}$}
		\label{fig:zigzagkerdctoHBC}
 	\end{figure}
 }

\subsection{Numeric inequalities and characterizations} \label{sec:numerics}

In Theorem \ref{thm:ddc+3}, we gave a numeric characterization of the $dd^c+3$ condition as an equality of cohomology dimensions. In this section we characterize this as the extremal case of an inequality valid for all compact complex manifolds, and derive some related numerical inequalities. These should be compared to the result of \cite{AnTo}, that for any compact complex manifold
	\[
	\sum_k \dim H^k_A(M)+ \dim H^k_{BC}(M) \geq 2 \sum_k b_k.
	\]
	Here $b_k:=\dim H^k(M)$ denotes the dimension of the de Rham cohomology, and equality holds if and only if the $dd^c$-condition holds. This will also follow from the results below. Again, all results remain valid if we let $\cA$ be an arbitrary bicomplex for which all quantities considered here are finite.

For any compact complex manifold, define the following numbers:
\begin{align*}
h_{BC} &= \sum_k \dim H^k_{BC}  & h_A &=\sum_k \dim H^k_{A}    \\
h_{\Ke d^c}  &= \sum_k \dim H^k(\Ke d^c) &   h_{\cA / \Img\,d^c} &= \sum_k \dim H^k(\cA / \Img \, d^c) \\  
h_\delb &= \sum_{\stackrel{p+q=k}{k\geq0}} \dim H^{p,q}_\delb  &  h_\del &= \sum_{\stackrel{p+q=k}{k\geq0}}\dim H^{p,q}_\del
\end{align*}

\begin{prop} \label{prop;numerics}
For any compact complex manifold $(M,J)$,
\[
h_{BC} + h_A \geq h_{\Ke d^c} + h_{\cA / \Img\,d^c} \geq h_\delb + h_\del \geq  2  \sum_k b_k.
\]
\begin{enumerate}
	\item The first inequality is equality if and only if there is a pure Hodge structure.
	\item The middle inequality is equality if and only if $E_2$-degeneration and purity defect $1$, i.e. only zigzags of length at most $3$.
	\item The last inequality is equality if and only if $E_1$-degeneration.
	\item  The first two inequalities are both equality if and only if pure Hodge and $E_2$-degeneration.
	\item The last two inequalities are both equality if and only if $d d^c + 3$. 
\end{enumerate}
\end{prop}
\begin{rmk}
	A characterization of the outermost equality $h_{BC}+h_A=2\sum b_k$ was also obtained in \cite{PoSteUb}.
\end{rmk}
\begin{proof}
Claim $(1)$ is in Theorem \ref{thm:purity1}, while claim $(3)$ is immediate. Claim $(4)$ follows from $(1)$ and $(2)$, while 
$(5)$ follows from $(2)$, $(3)$, and Theorem \ref{thm:ddc+3}. It remains to show $(2)$. This follows from an inspection of Figure \ref{fig:zzddc+3} and the observation that the expression $h_{\delb}+h_{\del}$ vanishes on squares and is equal to two on every zigzag, regardless of its length.
\end{proof}
\begin{rmk}
	On a compact complex manifold, by duality and real structure, one may replace the chain of inequalities by
	\[
	h_{BC}\geq h_{\Ke d^c}\geq h_{\delb}\geq \sum b_k,
	\]
	with the same characterizations of equalities.
\end{rmk}

\section{First examples of $dd^c+3$ manifolds and construction methods} \label{sec:methodsandexs}

In addition to all  $d d^c$-manifolds, there are numerous examples of $dd^c+3$-manifolds. 

\subsection{Complex surfaces} \label{ssec:surfaces}
For every compact complex surface $S$, the Fr\"olicher spectral sequence degenerates at the first page \cite{BaHu} and by Corollary \ref{cor: pdef leq n-1}, one has $\pdef(S)\leq 1$. Thus:

\begin{cor} \label{cor:surfaces}
	Any compact complex surface satisfies the $dd^c+3$-condition.
\end{cor}

One may also describe the entire bicomplex in detail: $E_1$-degeneration implies that there can be no even zigzags. $H^0_{\dR}(S)$, $H^2_{\dR}(S)$, $H^4_{\dR}(S)$ have a pure Hodge decomposition, \cite{BaHu}, so there are no odd length zigzags (other than dots) contributing to $b_0,b_2,b_4$. On the other hand, $H^1_{\dR}(S)$ (and by duality $H^3_{\dR}(S)$) admit a pure Hodge structure if and only if $b_1$ is even, which coincides with the K\"ahler case. 
If $b_1$ is odd, then $h^{0,1}=h^{1,0}+1$, and the first and third cohomologies have purity defect $1$, in the sense of Definition \ref{defn:puritydefect}. 

In fact, the decomposition of the complex valued differential forms $\cA(S)$ into indecomposables is as follows:
\[
\cA(S) \simeq_1\img{S000_2}^{\oplus b_0}\!\!\!\oplus~ \img {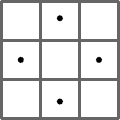}^{\oplus h^{0,1}} \!\!\!\!\oplus~\img{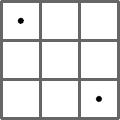}^{\oplus h^{2,0}}\!\!\!\!\oplus~ \img{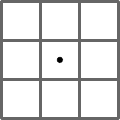}^{\oplus b_2-2h^{2,0}}\!\!\!\!\!\oplus~\img{S100_2}^{\oplus\varepsilon},
\]
Here $\varepsilon=0$ if and only if $S$ satisfies the $d d^c$-condition, and $\varepsilon=1$ otherwise.

 We remark that the entire long exact sequence from Theorem \ref{LES}, as well as the various groups considered here, can all be deduced from the decomposition into indecomposables above. Therefore, they are determined by the oriented topology of $S$.

\subsection{Twistor spaces}  \label{ssec:twspaces}
Let $M$ be a compact four-manifold with a self-dual Riemannian metric and denote by $Z(M)$ its twistor space. In \cite{EaSi}, it is shown that the Fr\"olicher spectral sequence of $Z(M)$ always degenerates at the second page. Furthermore, this second page $E_2$ is computed in terms of metric data on the base as follows:
\[
\begin{tikzpicture}[scale=0.5]
	
	\draw[->, thick] (0,0) -- (24.5,0) node[anchor=west] {$p$};
	\draw (0,2) -- (24,2);
	\draw (0,4) -- (24,4);
	\draw (0,6) -- (24,6);
	\draw (0,8) -- (24,8);
	
	\draw[->, thick] (0,0) -- (0,8.5) node[anchor=south]{$q$};
	\draw (4,0) -- (4,8);
	\draw (12,0) -- (12,8);
	\draw (20,0) -- (20,8);
	\draw (24,0) -- (24,8);
	
	\draw (2,-.5) node {0};
	\draw (8,-.5) node {1};
	\draw (16,-.5) node {2};
	\draw (22,-.5) node {3};
	
	\draw (-.5,1) node {0};
	\draw (-.5,3) node {1};
	\draw (-.5,5) node {2};
	\draw (-.5,7) node {3};

	\node (00) at (2,1) {$H^0(M;\CC)$};
	\node (01) at (2,3) {$H^1(M;\CC)$};
	\node (02) at (2,5) {$H^2_-(M;\CC)$};
	\node (03) at (2,7) {$0$};
	\node (10) at (8,1) {$0$};
	\node (11) at (8,3) {$H^0(M;\CC)\oplus H^2_+(M;\CC)$};
	\node (12) at (8,5) {$H$};
	\node (13) at (8,7) {$0$};
	\node (20) at (16,1) {$0$};
	\node (21) at (16,3) {$K$};
	\node (22) at (16,5) {$H_+^2(M;\CC)\oplus H^4(M;\CC)$};
	\node (23) at (16,7) {$0$};
	\node (30) at (22,1) {$0$};
	\node (31) at (22,3) {$H^2_-(M;\CC)$};	
	\node (32) at (22,5) {$H^3(M;\CC)$};
	\node (33) at (22,7) {$H^4(M;\CC)$};

\end{tikzpicture}
\]
Here $H^2_{\pm}(M;\CC)$ denote the spaces of (anti-)self-dual classes, and the definition of $H, K$ need not concern us.

Due to the lack of symmetry in the $E_2^{1,0}$ and $E_2^{0,1}$, we see that as soon as $b_1(M)\neq 0$, or $b_2^-(M)\neq 0$, the twistor space does not satisfy the $d d^c$-condition. However:

\begin{prop} \label{prop:twddc+3}
	Whenever $E_1(Z(M))=E_2(Z(M))$ and $b_2^-(M)=0$, the twistor space $Z(M)$ is $dd^c+3$.
\end{prop}
\begin{proof}
By the degeneration assumption, there are no even zigzags in any decomposition of $\cA(Z(M))$ into indecomposables. It remains to rule out the possibility of odd zigzags of length greater than three, which, for dimension reasons must have length five or seven. Length seven zigzags do not occur by Corollary \ref{cor: pdef leq n-1}. There are two possibilities for a length five zigzag on a three-fold:
\[
\img{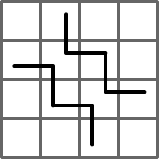}\quad\text{ and }\quad\img{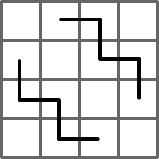}~.
\]
Any complex having one of these as a direct summand would have $E_2^{0,2}\neq 0$, and $E_2^{2,0}\neq0$, respectively, both of which are prohibited here.
\end{proof}
In \cite[Thm. 5.6]{EaSi}, Eastwood and Singer construct, for any $g\geq 0$, conformally flat metrics on $M=\#_g(S^1\times S^3)$  such that $E_1(Z(M))=E_2(Z(M))$. Since $b_2(S^1\times S^3)=0$, this implies:
\begin{cor}
	For any $g\geq 0$, there are metrics on $M=\#_g(S^1\times S^3)$ such that the associated twistor space $Z(M)$ is a $dd^c+3$-manifold.
\end{cor}

\subsection{Construction methods of $dd^c+3$-manifolds}  \label{ssec:constmethods}
The behavior of the bicomplex of differential forms, up to $E_1$-isomorphism, is known for several standard operations \cite{StDblBl}. From this we deduce many constructions which preserve the $dd^c+3$-condition.

\begin{prop} \label{thm:constrmthds}
The $dd^c+3$ condition has the following behavior under geometric constructions:
	\begin{enumerate}
		\item\label{itm: blowup ddc+3} A blow-up of a manifold $M$ along a smooth center $Z\subseteq M$ is $dd^c+3$ if and only if both $M$ and $Z$ are $dd^c+3$.
		\item\label{itm: ddbar x ddc+3} A product is $dd^c+3$ if and only if one factor is a $dd^c+3$-manifold and one is a $dd^c$-manifold.
		\item\label{itm: ddc+3 modification}The target of a holomorphic surjection $f:M\to N$ with $M$ $dd^c+3$ and $\dim M=\dim N$ is again $dd^c+3$.
		\item\label{itm: ddc+3 proj. bdls} Projectivized holomorphic vector bundles are $dd^c+3$-manifolds if and only if the base of the bundle is a $dd^c+3$-manifold.
	\end{enumerate}
\end{prop}

\begin{proof}
	Let $M$ be a complex manifold. Denote by $\tilde{M}$ the blow-up of $M$ in some smooth center $Z$ of codimension $d\geq 2$. Let $\cV$ be a holomorphic vector bundle of rank $r$ over $M$. Let $f:M\to N$ be a holomorphic surjection. In \cite{StDblBl}, it was shown that there are (chains of) $E_1$-isomorphisms
	\begin{align*}
		\cA(\tilde{M})&\simeq_1 \cA(M)\oplus_{i=1}^{d-1} \cA(Z)[i]\\
		\cA(\PP(\cV))&\simeq_1 \sum_{i=0}^{r-1} \cA(M)[i]\\
		\cA(M)&\simeq_1 \cA(N)\oplus \cA(M)/p^*\cA(N)
	\end{align*} Here, $\cA(Z)[i]$ denotes the complex shifted by degree $(i,i)$, i.e. $\cA(Z)[i]^{p,q}=\cA(Z)^{p-i,q-i}$. Since two bounded bicomplexes are $E_1$-isomorphic if and only if all zigzags have the same multiplicity, statements \ref{itm: blowup ddc+3}, \ref{itm: ddc+3 modification} and \ref{itm: ddc+3 proj. bdls} follow.

For statement \ref{itm: ddbar x ddc+3}, note that by the K\"unneth formula, there is an $E_1$-isomorphism $\cA(M\times N)\simeq \cA(M)\otimes \cA(N)$. If $M$ satisfies the $d d^c$-condition, we have $\cA(M)\simeq_1 D$, where $D$ has trivial differential, i.e. it is a direct sum of dots $D=\bigoplus D^{p,q}$. Then $\cA(M\times N)\simeq_1 \bigoplus_{p,q} D^{p,q}\otimes \cA(N)$ and the result follows because the tensor product of any bicomplex with a dot is isomorphic to a shift of the original bicomplex. Conversely, the purity defect is additive under products by Proposition \ref{prop: purdef of prod}, so $\pdef(M\times N)\leq 1$ implies that $\pdef(M)\leq 1$ and $\pdef(N)=0$, or vice versa. Furthermore, the Fr\"olicher spectral sequence of both $M$ and $N$ is a direct summand in that of $M\times N$, so if it degenerates on the product, it does on both factors.
\end{proof}
\begin{rmk}
	Using results of Meng one can generalize statements \ref{itm: ddc+3 modification} and \ref{itm: ddc+3 proj. bdls} (with a similar proof): 
	
	Statement \ref{itm: ddc+3 modification} holds more generally whenever there exists a closed current $T$ on $M$ of bidegree $(r,r)$ for $r=\dim M-\dim N$ such that $f_*T\neq 0$. This is the case for example for any map that admits a holomorphic section, c.f. \cite[§3]{Meng22}.
	
	 Statement \ref{itm: ddc+3 proj. bdls} holds more generally for relative flag varieties and any bundle the cohomology of which looks additively like a product with a $dd^c$-manifold (i.e. which satisfies a Leray-Hirsch type theorem), c.f. \cite[Prop. 3.3]{Meng22}.
\end{rmk}
In particular, condition \ref{itm: ddbar x ddc+3} holds when the center is any curve or surface. Because the equivalence relation determined by `bimeromorphism' is generated by blow-ups in smooth centers \cite{AKMW}, \cite{Wlod}, we obtain:
\begin{cor}
	The $dd^c+3$-property is a bimeromorphism invariant of compact complex manifolds in complex dimension at most four.
\end{cor}
Similarly, the $dd^c+3$ condition is a bimeromorphism invariant in any dimension if and only if submanifolds of $dd^c+3$-manifolds are again $dd^c+3$-manifolds.
\begin{rmk} \label{rmk;rank-delta-Bimer-Inv-Leq4}
	Since the connecting map $\delta_k$ in the long exact sequence vanishes on $dd^c+3$-manifold, one can generalize the last Corollary to the statement that the rank of $\delta_k$ is a bimeromorphism invariant in complex dimension at most four.
	\end{rmk}

\subsection{Stability under deformations} \label{ssec:stabilitydef}

In this subsection we show that purity defect behaves upper semi-continuously, so that the $dd^c+3$-condition is stable under small deformations. Under slightly stronger assumptions, which includes the cases of compact surfaces and Vaisman manifolds, then entire bicomplex is unchanged under small deformations, up to $E_1$-isomorphism, Proposition \ref{prop: def stability for strong ddc+3}.

\begin{thm}\label{thm: semi-cont FcapFbar}
	Let $M$ be a compact complex manifold with degenerate Fr\"olicher spectral sequence $E_1(M)=E_\infty(M)$. For any small deformation $\pi:\mathcal{M}\to \Delta_\epsilon(0)=:B$ with $M=M_0=\pi^{-1}(0)$, the dimension 
	\[
	f^{p,q}_k(t):=\dim F^pH_{\dR}^k(M_t)\cap \bar{F}^qH_\dR^k(M_t)
	\]
	behaves upper semi-continuously, i.e. for any $t$ sufficiently close to $0$ one has:
	\[
	f^{p,q}_k(0)\geq f^{p,q}_k(t).
	\]
\end{thm}

We learned an essential part of the argument below, namely treating $F^pH_{\dR}^k \cap \bar{F}^qH_\dR^k$ as the intersection of vector bundles,  from a talk of Chi Li, c.f. \cite{Li}, following Voisin \cite{Voisin}.

\begin{proof}
First, we recall the well known argument (see e.g. \cite{Voisin}) that for $t$ sufficiently close to $0$, the Fr\"olicher spectral sequence of $M_t$ degenerates and the Hodge numbers are the same as those of $M_0$: Choosing a smooth family of hermitian metrics on the fibres $M_t$, the Hodge numbers may be computed via the $\delb$-Laplacian which is an elliptic operator that varies smoothly in $t$. Therefore the Eigenvalues vary continuously and in particular the dimension of the kernel (i.e. the Hodge numbers) can only drop when passing from $M_0$ to a nearby fibre:
\[
h^{p,q}_\delb(M_0)\geq h^{p,q}(M_t).
\]
On the other hand, 

\[
b_k(M_0)=\sum_{p+q=k}h^{p,q}(M_0)\geq \sum_{p+q=k}h^{p,q}(M_t)\geq b_k(M_t)=b_k(M_0)
\]
so one has to have equalities everywhere.

Now, for any $t$, consider the inclusion of complexes $F^p\cA(M_t)\subseteq \cA(M_t)$ defined by $F^p\cA(M_t)=\bigoplus_{r\geq p} \cA^{r,s}(M_t)$. The induced map on cohomology has image $F^pH_\dR^k(M_t)$ and kernel 
\[
\frac{\operatorname{\Imt} d\cap F^p\cA^k_{M_t}}{d(F^p\cA^k_{M_t})}.
\]
Thus, the induced map on cohomology is injective if and only if the differential $d$ is strict, which in turn is known to be equivalent to degeneration of the Fr\"olicher spectral sequence \cite[1.3.2]{DeHII}. Therefore, for small $t$, we have an identification $F^pH_\dR^k(M_t)=H^k(F^p\cA(M_t))$.

 Now, because $F^p\cA(M_t)$ is an elliptic complex (see e.g. \cite{St22}), or because their collection over all $t$ forms a resolution of the truncated complex of coherent sheaves of relative holomorphic forms \[0\to\Omega^p_{\cM/B}\to\Omega_{\cM/B}^{p+1}\to\cdots,\] the cohomology groups
$H^k(F^p\cA(M_t))$ form a vector bundle on the base as soon as their dimensions are constant. On the other hand, because the dimension of a filtered vector space is the same as that of its associated graded, we find that
\[
\dim F^pH^k_\dR(\cA(M_t))=\sum_{r+s=k, r\geq p} h^{r,s}(M_t)
\]
is constant indeed. 

In summary, for sufficiently small $t$, we have proved that $\{F^pH_\dR^k(M_t)\}$ form a complex  vector subbundle of the vector bundle $\{H_{\dR}^k(M_t)\}$. The same holds for the conjugate filtration $\bar{F}^qH_\dR^k(M_t)$. But the dimension of an intersection of two  vector subbundles behaves upper semi-continuously. 
\end{proof}
\begin{cor}
	For any $n$-dimensional compact complex manifold $M$ with $E_1(M)=E_\infty(M)$, the dimension of the spaces in the $3$-space decomposition \cite[Thm 4.8]{StDblBl}
	\[
	H_\dR^n(M)=H^{n,0}(M)\oplus \left(F^{1}H^n_\dR(M)\cap\bar F^1H^n_\dR(M)\right)\oplus H^{0,n}(M)
	\]
	is constant under small deformations.
\end{cor}

\begin{cor}\label{cor: pdef semcont}
	In the situation above, the purity defect behaves upper semi-continuously, in the following sense: For sufficiently small $t$,
	\[
	\pdef(M_0)\leq k \quad \Longrightarrow \quad  \pdef(M_t)\leq k
	\]
\end{cor}

\begin{rmk}
	Degeneration of the Fr\"olicher spectral sequence in the central fibre is necessary to obtain the conclusion of Corollary \ref{cor: pdef semcont}. In fact, the Iwasawa manifold has purity defect $0$, but it admits small deformation with purity defect $1$ and $2$, see \cite[§9.1]{SteLR}.
\end{rmk}

\begin{cor} \label{cor:ddc+3stable} 
	The $dd^c+3$-condition is stable under small deformations.
\end{cor}
\begin{rmk}
	The $dd^c+3$-condition introduced here should be compared with the page-$1$-$\del\delb$-condition introduced in \cite{PoSteU}, \cite{PoSteUb}. For the latter, one keeps purity, but relaxes the degeneration step of the Fr\"olicher spectral sequence. That condition, too, enjoys some stability under geometric constructions as in Prop. \ref{thm:constrmthds}. However, it is in general not stable under small deformations.
\end{rmk}

Given the fact that the dimensions of $H_{\dR}^k$ are constant under small deformation, a natural question is whether, given degeneration at $E_1$, the dimensions of the spaces $F^pH^k_\dR\cap \bar F^q H^k_{\dR}$ may actually change under small deformations. In general, the answer is yes.
\begin{ex}
	Consider a family of complex manifolds $\{M_t\}$ such that the central fibre is a type $(iii.a)$ deformation of the Iwasawa manifold and the nearby fibres are type $(iii.b)$ deformations (see \cite[§3.2.1.2]{AngellaBook} for the definition of these deformations). Then the central fibre has purity defect $2$ and degenerate Fr\"olicher spectral sequence, but the nearby fibres have purity defect $1$, see \cite[§9.1]{SteLR}
\end{ex}
However, in certain situations, the spaces $F^pH_{\dR}^k\cap \bar F^qH_{\dR}^k$ actually do have constant dimension:

\begin{prop}\label{prop: def stability for strong ddc+3}
Let $M$ be a compact complex manifold with $E_1(M)=E_\infty(M)$,  such that
\begin{itemize}
\item[($\ast$)] for any $k$, there exists an $r(k)$ such that $\gr_{F_{tot}}^{d}H^k(M)=0$ unless $d=r(k),r(k)-1$.
\end{itemize}
Then, any sufficiently small deformation $M_t$ of $M=M_0$ has the same $E_1$-isomorphism type as $M$, i.e. 
	for all $t$ sufficiently small:
 \begin{enumerate}
 \item The bicomplex $\cA(M_t)$ has the same zigzag multiplicities as $\cA(M_0)$,
	\item For any cohomological functor $H$ (e.g. $H_{BC},H_A,H_{\delb},..$), $H(M_t)\cong H(M_0)$.
 \end{enumerate}
 
\end{prop}
The condition $(\ast)$ visually says that the odd length zigzags appearing in the bicomplex are `not too distinct' in the following sense: Order the odd-length zigzags (up to translation) in an ascending way by their length, where we associate negative length to zigzags with incoming outermost arrows, i.e.
\[
...\leq\text{reverse L's}\leq \text{dots}\leq \text{L's}\leq...
\]

 Then the condition $(\ast)$ says that at most two directly adjacent zigzags types may contribute to de Rham cohomology in any given degree. For example, if the purity defect is one, it says there are (at most) only dots and L's or (at most) only dots and reverse L's in any given degree.

\begin{rmk} \label{rmk: SandVare*}
This condition $(*)$ is satisfied for compact complex surfaces as can be seen from the explicit description of their bicomplexes above, and is also satisfied for Vaisman manifolds, as will follow from Theorem \ref{thm: dc-str Vaisman} below (c.f. Corollary \ref{cor: E1isotype of Vaisman constant under small defs}). One may also prove an analogue of Prop. \ref{thm:constrmthds} for condition $(\ast)$ instead of the $dd^c+3$-condition. 
\end{rmk}

\begin{proof}
It suffices to show the multiplicities of all zigzags are constant for $t$ close to $0$. There are no even zigzags on $M_0$ or nearby fibres by degeneration of the Fr\"olicher spectral sequence, as in the first part of the proof of Theorem \ref{thm: semi-cont FcapFbar}. The odd zigzags are counted by the refined Betti numbers $b_k^{p,q}(M_t)$. Thus we have to show that the numbers $b_k^{p,q}(M_t)$ are constant for $t$ close to $0$. 

First we note that condition ($\ast$) has to hold for nearby fibres as well. In fact, when $F_{tot}^rH^k(M_0)=0$ for some $r$ then also $F_{tot}^rH^k(M_t)=0$ for all nearby fibres by Theorem \ref{thm: semi-cont FcapFbar}. By duality, the same implication holds for the condition $F_{tot}^rH^k=H^k$.

When $p+q=r(k)$, we have
\[
b_k^{p,q}=\dim F^pH^k\cap \bar F^q H^k
\]
by assumption, and we have seen that this number varies upper semi-continuously in Theorem \ref{thm: semi-cont FcapFbar}. On the other hand, $b_k^{p,q}(M)=b_{2n-k}^{n-p,n-q}(M)$ and so also the numbers for $p+q-k=r(k)-1$ vary upper semi-continuously. Now
\[
b_k(M_t)=\sum_{p+q\in\{r(k),r(k)+1\}}b_k^{p,q}(M_t)\leq \sum_{p+q\in\{r(k),r(k)+1\}}b_k^{p,q}(M_0)=b_k(M_0)=b_k(M_t). 
\]
\end{proof}

\section{Vaisman manifolds} \label{sec:Vaisman}
 
A Vaisman manifold will mean a compact complex manifold with Hermitian metric which is locally conformally K\"ahler (LCK) and has parallel Lee form, \cite{Vais}, \cite{VaisGH}. Recall the locally conformal K\"ahler condition is equivalent to the fundamental form $\omega$ satisfying $d\omega=\theta\wedge\omega$ for a closed real $1$-form, called the Lee form, and the parallel condition is that $\nabla\theta=0$ with respect to the Levi-Cevita connection. As is customary, we will assume $\theta\neq 0$ in the following to exclude the K\"ahler case from the discussion. 

\begin{ex}
	The Hopf manifold $(\CC^{n+1}\setminus\{0\})/\lambda^\Z$ for some $\lambda\in\CC^\ast\setminus S^1$ carries the Vaisman metric $\frac 1{\|z\|^2}\sum_{i=1}^{n+1}dz_id\bar z_i$. More generally, take any projective manifold with a negative line bundle $L$ and consider $V:=(L\setminus\{s_0\})/\lambda^\Z$ where $s_0$ denotes the zero section. Then $V$ carries a Vaisman metric \cite{VaisLCKGCK}.
\end{ex}
In fact, this example gives a good (local) picture of the complex structure of a general Vaisman manifold. We refer to \cite{OV} and \cite{OV4} for a general discussion of the structure of compact Vaisman manifolds.

\subsection{The $E_1$-isomorphism type of a Vaisman manifold}
Denote the bigraded components of the closed Lee form by $\theta=\theta^{1,0}+\theta^{0,1}$. Then $d\theta^{1,0}=\delb\theta^{1,0}=-d\theta^{0,1}=-\del\theta^{0,1}$, and setting $\omega_0:=d^c\theta=-2i\del \theta^{0,1}$, we have $\omega_0=\omega-\theta\wedge J\theta$.

The dual vector fields $X_\theta$ and $X_{J\theta}$ are holomorphic, Killing, and generate a group that acts by holomorphic isometries.
Let $\cA^{inv}(M)$ denote the complex of invariant forms under the group action, with subcomplex the basic forms $\cA_B(M)$,
i.e. those in the kernel of $\iota_{X_\theta}$ and $\iota_{X_{J\theta}}$, as well as the kernel of the Lie derivatives 
$\mathcal{L}_{X_\theta}$ and $\mathcal{L}_{X_{J\theta}}$.  The subspace of $d$-harmonic basic forms,  $\Hh_B$, behaves in much that same way as the forms on a K\"ahler manifold, having a Lefschetz decomposition given by the operator 
$L$ given by wedging with $\omega_0$. We refer to \cite{OV2} for a more thorough review of the operators mentioned here and their relations.

\begin{prop}\label{prop: Vaisman-dc via harmonics} (\cite{Tsu},\cite{IshKas})
	With notations as above, the subspace
	\[
	\Hh_B\otimes \Lambda\langle \theta^{0,1},\theta^{1,0}\rangle\subseteq \cA(V)
	\]
	is a $d$-subcomplex and the inclusion is an $E_1$-isomorphism.
\end{prop}

Our goal is to describe the structure of the bicomplex $\Hh_B\otimes \Lambda\langle \theta^{0,1},\theta^{1,0}\rangle$ in terms of indecomposables. We will reduce it to an algebraic computation below in the following way. For any $p+q\leq n$ denote by 
\[
P_{p,q}:=\Ke L^{n-p-q+1}\subseteq \Hh_B
\]
the space of primitive harmonic $(p,q)$-forms. Consider $\Hh_B$ as a $\CC[L]$-module, and let $D_{p,q}$ be the $\CC[L]$-submodule generated by $P_{p,q}$, i.e. $D_{p,q}:=P_{p,q}[\omega_0]/\omega_0^{n-p-q+1}$. Writing $S_{p,q}:=D_{p,q}\otimes \Lambda\langle\theta^{1,0},\theta^{0,1}\rangle$, we have 
\[
\Hh_B\otimes \Lambda\langle \theta^{0,1},\theta^{1,0}\rangle=\bigoplus_{p+q\leq n} S_{p,q}.
\]

\begin{thm}[The bicomplex of a Vaisman manifold]\label{thm: dc-str Vaisman}
	Let $V$ be a compact Vaisman manifold of dimension $n+1$.
	The inclusion of bicomplexes
	\[
	\bigoplus_{p+q\leq n} S_{p,q}\subseteq \cA(V)
	\]
	is an $E_1$-isomorphism. 
	Every $S_{p,q}$ is, as a bicomplex, a tensor product of the form 
	\[
	S_{p,q}=P_{p,q}\otimes \Lambda\langle\theta^{1,0},\theta^{0,1},\omega_0\rangle/(\omega^{n-p-q+1}_0).
	\]
	The space of primitive basic harmonic forms $P_{p,q}$ has zero differential, i.e. it is a direct sum of dots. The second factor decomposes as follows into indecomposable bicomplexes:
	\begin{enumerate}
		\item If $k=n$, there are four dots, with no non-zero differentials:
		\[
		\xymatrix{
			\langle\theta^{0,1}\rangle \ar@{}[r]|\oplus &  \langle \theta^{1,0}\theta^{0,1}\rangle\\
			\CC \ar@{}[u]|\oplus \ar@{}[r]|\oplus &  \langle \theta^{1,0}\rangle \ar@{}[u]|\oplus
		}
		\]
		
		\item If $k<n$, there are two dots, two length three zigzags, and (if $k<n-2$) several squares:
		{\small
			\[
			\xymatrix@=15pt{
				& & & & \langle\theta^{0,1}\omega_0^{n-k}\rangle & \langle\theta^{0,1}\theta^{1,0}\omega_0^{n-k}\rangle \\
				& & & & \langle\theta^{1,0}\theta^{0,1}\omega_0^{n-k-1}\rangle\ar[r]^-\del \ar[u]^-\delb & \langle\theta^{1,0} \omega_0^{n-k} \rangle \\
				&  & \langle\theta^{0,1}\omega_0^{j+1}\rangle\ar[r]^-\del  \ar @{} [dr] |{\oplus_{j=0}^{n-k-2}} &  \langle\omega_0^{j+2}\rangle \ar@{}[ur]|\iddots & & \\
				& &  \langle\theta^{1,0}\theta^{0,1}\omega_0^j\rangle\ar[r]_-\del \ar[u]^-\delb &\langle\theta^{1,0}\omega_0^{j+1} \rangle \ar[u]_-\delb  & & \\
				\langle\theta^{0,1}\rangle \ar[r]^-\del &  \langle \omega_0 \rangle \ar@{}[ur]|-\iddots & & & &  \\
				\CC & \langle \theta^{1,0}\rangle \ar[u]_-\delb & & & &  \\
			}
			\]
		}
	\end{enumerate}
\end{thm}

\begin{proof} 
	Only the statement about the second factor of $S_{p,q}$ still requires proof. For the first claim, with $k=n$, the bicomplex is 
	\[
	\Lambda\langle\theta^{1,0},\theta^{0,1},\omega_0\rangle/(\omega_0)= \Lambda\langle\theta^{1,0},\theta^{0,1}\rangle,
	\] 
	and the relation $\del\theta^{0,1}=-\delb\theta^{1,0}=\frac i 2 \omega_0$ implies all differentials are zero.
	
	For the general case, $k<n$, we have $\del\theta^{0,1}=-\delb\theta^{1,0}=\frac i 2 \omega_0$ and $\delb\theta^{0,1}=\del\theta^{1,0}=0$, giving the length three zigzag
	\[
	\xymatrix@=15pt{
		\langle\theta^{0,1}\rangle \ar[r]^-\del &  \langle \omega_0 \rangle \\
		& \langle \theta^{1,0}\rangle \ar[u]_-\delb   
	}
	\]
	in total degrees $1$ and $2$. The same relations also give the squares in the statement, with all other differential on these spaces zero, again since $\delb\theta^{0,1}=\del\theta^{1,0}=0$. Finally, the length three zigzag
	\[
	\xymatrix@=15pt{
		\langle\theta^{0,1}\omega_0^{n-k}\rangle &  \\
		\langle\theta^{1,0}\theta^{0,1}\omega_0^{n-k-1}\rangle\ar[r]^-\del \ar[u]^-\delb & \langle\theta^{1,0} \omega_0^{n-k} \rangle \\
	}
	\]
	follows from the same relations and the fact that $\omega_0^{n-k+1}=0$, which also implies that
	$\del \theta^{0,1}\theta^{1,0}\omega_0^{n-k} = \delb \theta^{0,1}\theta^{1,0}\omega_0^{n-k}=0$.
\end{proof}

Since the bicomplex of a Vaisman manifold satisfies condition $(5)$ of Theorem \ref{thm:ddc+3} we have:

\begin{cor} \label{cor:Visddc+3}
	If a compact complex manifold admits a Vaisman metric, then it satisfies the $dd^c+3$-condition. 
\end{cor}

In particular, the Fr\"olicher spectral sequence degenerates at the first page, which was also shown in \cite{Tsu}. By results from section \ref{sec:purity}, we also have:

\begin{cor}  \label{cor:Vpd=1}
	If a compact complex manifold admits a Vaisman metric, then it has purity defect $1$.
\end{cor}
\begin{cor} \label{cor:Vphmiddle}
	The middle cohomology of a compact Vaisman manifold carries a pure Hodge structure.
\end{cor}

\begin{cor}  \label{cor:VHBC}
	For any compact Vaisman-manifold $V$ of dimension $n+1$, the Bott-Chern and Aeppli cohomologies up to middle degree can be computed as follows:
	\[
	H_{BC}^{p,q}(V)\cong \begin{cases} 
		P_{p,q}\oplus \omega_0P_{p-1,q-1} & \text{if }p+q\leq n\\
		\theta^{1,0} P_{p-1,q}\oplus \theta^{0,1} P_{p,q-1} &\text{if } p+q=n+1
	\end{cases} 
	\]
	\[
	H_A^{p,q}(V)\cong 
	\begin{cases}
		P_{p,q}\oplus \theta^{1,0}P_{p-1,q}\oplus \theta^{0,1}P_{p,q-1}&\text{if }p+q\leq n\\
		\theta^{1,0}P_{p-1,q}\oplus \theta^{0,1} P_{p,q-1}&\text{if }p+q=n+1
	\end{cases}
	\]
\end{cor}

Note that the groups above middle degree are determined by duality. They can also be written down explicitly using the same method of proof.

\begin{proof}
	Using the notation introduced before Theorem \ref{thm: dc-str Vaisman}, we have:
	\[
	\cA(V) \simeq_1 \bigoplus_{r,s\in\Z}S_{r,s}
	\]
	by  Proposition \ref{prop: Vaisman-dc via harmonics}.
	Now, \cite[Cor. 13]{StStrDbl} states that an $E_1$-isomorphism induces an isomorphism on $H_{BC}$, so by the Definition of $S_{r,s}$ and the fact that $P_{r,s}$ is a complex with trivial differentials, we have:
	\[
	H^{p,q}_{BC}(V)=\bigoplus_{r,s} P_{r,s}\otimes H^{p,q}_{BC}\left(\Lambda\langle \theta^{0,1},\theta^{1,0},\omega_0\rangle/(\omega_0^{n-(r+s)+1})\right).
	\]
	The result now follows by Theorem \ref{thm: dc-str Vaisman}, as the groups $H_{BC}$ on any zigzag are known to be computed by 
	the dots (here $P$) and the ``tips'', i.e. the spaces $X$, $U$, and $T$ in diagrams such as those below:
	\[
	\xymatrix@=15pt{
		W \ar[r]^-\del & X \\
		& Y  \ar[u]_-\delb 
	}
	\quad \quad\quad \quad 
	\xymatrix@=15pt{
		U&  \\
		Z \ar[u]^-\delb \ar[r]^\del & T.
	}
	\]
	The proof for Aeppli cohomology is the same, except $H_{A}$  is computed by 
	the dots (again $P$) and the spaces $W$, $Y$ and $Z$ in the diagrams above.
\end{proof}

Oeljeklaus-Toma (OT) manifolds are manifolds associated with number fields that have $s\geq 1$ real and $t\geq 1$ pairs of distinct conjugate complex embeddings, together with the choice of appropriate subgroups of the group of totally real units. We refer to \cite{OT} for their definition and more details. 
\begin{rmk} (LCK does not imply $dd^c+3$) OT manifolds of type $(s,1)$ with $s\geq 2$ are LCK \cite[p.169]{OT}, but not $dd^c+3$. In fact, the computation in \cite[Cor. 9.6]{SteLR}, their bicomplex always contains zigzags of length $2s+1$. For example, for $M$ an OT manifold of type $(2,1)$, one has
	\[
	\cA(M)\simeq_1\img{S000_3}\oplus\img{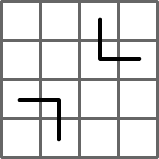}^{\oplus{2}}\oplus \img{S200_3}.
	\]	
	
\end{rmk}
Generalizing the calculation in the previous remark allows to obtain the following Corollary, previously proved by Hisashi Kasuya \cite{KasVais} via a different route:
\begin{cor}
	Oeljeklaus-Toma manifolds of type $(s,t)$ with $s\geq 2$ are never Vaisman.
\end{cor}
\begin{proof}
According to \cite{OtiTo}, the Fr\"olicher spectral sequence degenerates on all Oeljeklaus-Toma manifolds and the purity defect of an OT manifold of type $(s,t)$ is equal to $s$ (this follows from \cite[Thm. 9]{ADOS}).
\end{proof}

Along a small deformation of a compact complex manifold with degenerate Fr\"olicher spectral sequence, the Hodge numbers remain constant. In the case of Hopf manifolds, the Hodge numbers determine the multiplicities of all zigzags combinatorially, see \cite{StStrDbl}. Therefore, the multiplicities of all zigzags stay the same under small deformations. The following Corollary is a generalization of this fact to all Vaisman manifolds and all cohomological functors, which follows directly from Remark \ref{rmk: SandVare*} and Proposition \ref{prop: def stability for strong ddc+3}:

\begin{cor}\label{cor: E1isotype of Vaisman constant under small defs}
	Every small deformation $V_t$ of a compact Vaisman manifold $V_0$ has the same $E_1$-isomorphism type, i.e. 
	for all $t$ sufficiently small:
	\begin{enumerate}
	\item The bicomplex $\cA(V_t)$ has the same zigzag multiplicities as $\cA(V_0)$,
	\item For any cohomological functor $H$ (e.g. $H_{BC},H_A,H_{\delb},..$),  $H(V_t)\cong H(V_0)$.
	\end{enumerate}
\end{cor}

\subsection{Vanishing of higher multiplicative operations} \label{Vanishing of higher multiplicative operations}

In this section we record a Vaisman analogue of the formality result of \cite{DGMS} for K\"ahler manifolds. Namely, we show the vanishing of certain higher cohomology operations on compact Vaisman manifolds. This is to be expected by the close relationship between Vaisman and Sasakian manifolds, and the results of \cite[Prop. 4.4]{BFMT} showing that compact Sasakian manifolds have vanishing quadruple and higher Massey products. The latter follows from an algebraic statement proved in \cite[Prop. 4.5]{BFMT}, and essentially the same argument shows:

\begin{lem}
	Let $B=(\bigoplus_{i=0}^{2n+1}B^i,d=0)$ be a (connected) cdga with trivial differential, $\omega\in H$ of degree $2$ and consider an elementary extension $A=(B\otimes \Lambda(y),d)$, with $dy=\omega$. Assume that any cohomology class in $H(A)$ of degree at most $n$ has a representative in $B$ and that any cohomology class in degree at least $n+2$ has a representative in $By$. Then all Massey products $\langle a_1,...,a_k\rangle\in H(A)$ with $k\geq 4$ and no $a_i$ of degree $n+1$ contain zero.
\end{lem}

\begin{cor} \label{VMPvanish}
	On a compact Vaisman manifold $V$ of dimension $n+1$, a Massey product $\langle a_1,...,a_k\rangle\in H(V)$  with $k\geq 4$ contains zero, provided that no $a_i$ has degree $n+1$.
\end{cor}

We do not know whether there exists a compact Vaisman manifold with a nontrivial quadruple (or higher) Massey product.

\begin{proof}
As seen before, Vaisman manifolds $V$ have a model $B\otimes\Lambda(J\theta)$, with $B=H\otimes\Lambda(\theta)$, where $H=H_B(V)$ denotes the basic cohomology. This satisfies the required conditions by the explicit formulae for the cohomology of Vaisman manifolds: In fact, denoting by $P_B^k\subseteq H^k $ the primitive part of the basic cohomology, one has\[
H_{\dR}^k(V)=\begin{cases}
	P_B^k\oplus P_B^{k-1}\theta&\text{if }k\leq n\\
	P_B^n\theta\oplus P_B^nJ\theta&\text{if }k=n+1\\
	(P_B^{n-l}\omega_0^{l}\oplus P_B^{n-(l-1)}\omega_0^{l-1}\theta)J\theta&\text{if }k=n+1+l\geq n+2
\end{cases}
\] 
(see for instance \cite{Kash}, \cite{VaisGH}, \cite{OV2} or Theorem \ref{thm: dc-str Vaisman}).
\end{proof}

To establish the second vanishing result in this section, Proposition \ref{C3prop}, we review some background material on $C_\infty$-algebras, $C_n$-algebras, and their morphisms, which are due to Kadeishvili \cite{Kad2}.  These are commutative analogues of $A_\infty$-algebras, etc.,  due to  \cite{Stash}. We use the unshifted sign conventions, for example of \cite{Markl2}, and the Koszul rule for signs is implicit. 
 
 An \textit{$A_\infty$-algebra} on a graded vector space $A = \{A^n\}_{n \in\Z}$ is a 
collection of linear maps $m_k: A^{\ot k} \to A$ of degree $2-k$, for $k \geq 1$ such
that
\[
[d, m_k] = \sum_{\stackrel{ j+\ell =k+1; \, j,l\geq 2}{1 \leq i \leq j}} (-1)^{i(\ell +1)+k} m_j \left( \id^{\ot i-1} \ot m_\ell \ot  \id^{\ot j-i} \right),
\]
where the left hand side uses differential in the complex $Hom(A^{\ot k},A)$. The equations imply $d := m_1$  satisfies $d^2=0$, and that $d$ is a derivation of the product $m_2$. Also, $m_3$ is a chain homotopy for the associativity condition. 

A \textit{$C_\infty$-algebra} is an $A_\infty$-algebra for which each $m_k \in Hom(A^{\ot k},A)$ vanishes on the image of the shuffle product of the tensor co-algebra of $A$ shifted down by one.  We refer the reader to \cite{Markl} for the precise definition of these, which will not be needed here.  By definition, a $C_n$-algebra is a $C_\infty$-algebra such that $m_k=0$ for all $k>n$. A cdga is precisely a $C_2$-algebra. A $C_\infty$-algebra is called \emph{minimal} if $m_1=0$. 

If $(A, m_k)$ and $(A',m_k')$ are $A_\infty$-algebras, an $A_\infty$-morphism from $A$ to $A'$ is a collection of 
linear maps $f_k:A^{\ot k} \to A'$ of degree $1-k$ such that for every $k \geq 1$
\[
[d,f_k] + \sum_B (-1)^{\epsilon} m'_j( f_{r_1} \ot \cdots \ot f_{r_j}) =
\sum_{\stackrel{ j+\ell =k+1; \, j,l\geq 2}{1 \leq i \leq j}} (-1)^{i(\ell +1) + k} f_j (\id^{\ot i-1} \ot m_\ell \ot  \id^{\ot j-i} )
\]
where  
\[
B= \{j,r_1,\ldots,r_j|  2 \leq j \leq k, r_1,\ldots, r_j \geq 1, r_1 +\ldots +  r_j=k\},
\]
\[
\epsilon:= \epsilon(r_1,\ldots,r_j) = \sum_{1 \leq \alpha < \beta \leq s} r_\alpha( r_\beta +1).
\]

A $C_\infty$-morphism is a morphism $\{f_k\}$ of $A_\infty$-algebras such that each map $f_k$ vanishes on shuffles.
An $A_\infty$ or $C_\infty$-morphism is a \emph{quasi-isomorphism} if  $f_1: A \to A'$ induces an isomorphism 
in cohomology of the complexes $(A, m_1)$ and $(A', m'_1)$. 

A $C_\infty$-structure is called \textit{unital} if there's an $m_1$-closed element $1 \in A^0$ which is a unit for the product $m_2$, and $m_k$ vanishes for $k>2$ whenever $1$ is inputted. A morphism $\{f_k\}$  of unital $C_\infty$ algebras means $f_1$ preserves units, and $f_k$ vanishes for $k>2$ whenever $1$ is inputted. 

We now come to the main transfer theorem for transfer of $C_\infty$-structures \cite{CG}. This result has a long list of
antecedents, e.g. \cite{Markl2}, \cite{KS}, \cite{Merk}, \cite{Kad4}, \cite{Markl}, \cite{Kad}, which vary in their level of generality and explicitness of formulas and signs. 

We restrict to the case where $(A,d)$ is a unital cdga, and suppose we have  a \emph{contraction}, namely chain maps
$\pi: A \to H$, $i: H \to A$, with $\pi \circ i = \id_H$ and $i \circ \pi - id = [d,h]$ for some homotopy $h: A^* \to A^{*-1}$ on the complex $(A,d)$. We assume the follow side conditions hold
\[
h^2=0, \quad \quad h \circ i =0, \quad  \quad \pi \circ h = 0,
\]
which can always be arranged, and will hold in the applications below.

\begin{thm} (\cite{CG})  \label{thm:transf}
For any unital cdga $(A,d,m)$ with a contraction $(A,H^*(A),\pi,i)$ satisfying the above side conditions, 
 there is a minimal unital  $C_\infty$-algebra $(H^*(A),m_k)$ given inductively, for $k \geq 2$, by
\[
m_k = \pi \circ p_k \quad \textrm{where} \quad p_k = m \left (\sum_{j=1}^{k-1} (-1)^{k} h p_j \ot
h p_{k-j} \right),
\]
where $h p_1:= i$. Furthermore, there is a  unital $C_\infty$-quasi-isomorphism 
$f_k:A^{\ot k} \to H^*(A)$ given by  $f_k = h \circ p_k$, for $k\geq1$.
\end{thm}

Note that $m_2 = \pi \circ m \circ (i \ot i)$ is simply the transport of the product by the chain equivalence, and on 
elements $x=i(a), y= i(b), z=i(w)$, 
\[
m_3(x,y,z) = -\pi ( h(xy)z + (-1)^{|x|}  xh(yz)). 
\]

 The recursive formulas in the statement above can also be expressed as ``sum over trees'' formulas, which are perhaps easier to visualize. The following lemma gives sufficient conditions for the induced $C_\infty$-algebra on  $H^*(A)$ to be a $C_3$-algebra.
 
\begin{prop} \label{C3prop}
Let $(A,d)$ be a unital cdga such that 
\[
A \cong B \otimes  \Lambda x
\]
as unital algebras, where $B$ has trivial differential, $\mathrm{deg}(x)$ is odd, and $dx \in B$.  Then $(A,d)$ is 
quasi-isomorphic to a minimal unital $C_3$-algebra, via a unital $C_\infty$-quasi-isomorphism
$f_k:A^{\ot k} \to H^*(A)$ satisfying $f_k = 0$ for $k \geq 4$.
\end{prop}

\begin{proof}
Using the notation of Theorem \ref{thm:transf}, it suffices to show that $p_k = 0$ for all $k \geq 4$.

As complexes, we have $A= B\oplus Bx$, and $Bx$ is an ideal in $A$ satisfying $(Bx)^2=0$.
Note that $\Imt d\subseteq B$ and $\Ke d=B\oplus \Ke d\cap Bx$. This allows to choose a splitting $A=\Imt d \oplus L\oplus H$, where $L\subseteq Bx$, $d:L\to \Imt d$ is an isomorphism and $d|_H=0$. Choose $i: H \to A$ to be the inclusion, that represents the cohomology of $A$, with projection $\pi: A \to H$ 
and define a contracting homotopy $h: A \to A$ to be a projection onto $\Imt(d)$, followed by $d^{-1}:\Imt d\to L$.
By construction,  $\Imt h$ is contained in the ideal $Bx$, and satisfies $h(Bx) = (\Imt h)^2 = 0$.

Therefore, for $k=3$ we have 
\[
p_3 = m \left (i \ot h p_2 \right) + m \left ( h p_2 \ot i \right) \subseteq Bx,
\]
so that $p_4=0$, since
\[
m(i \ot h p_3) = m(h p_2 \ot h p_2 ) = m( h p_3 \ot i )=0.
\]
Similarly, for all $k \geq 4$ and every $1 \leq j <k$, $m \left ( h p_j \ot h p_{k-j} \right) =0$, so $p_k = 0$.
\end{proof}

According to \cite{T}, the differential forms of a compact Sasakian manifold have a model which satisfies the condition of Proposition
\ref{C3prop}. 

\begin{cor} \label{cor;SC3}
 For any compact Sasakian manifold the differential forms are $C_\infty$-quasi-isomorphic to a minimal unital $C_3$-algebra.
 \end{cor}
 
 Additionally we have:

 \begin{cor} \label{cor:VC3}
For any compact Vaisman manifold the differential forms are $C_\infty$-quasi-isomorphic to a minimal unital $C_3$-algebra.
\end{cor}

\begin{proof} 
A real model of the complex $(\cA(V), d)$ is given by
	\[
	H \otimes \Lambda\langle \theta^{0,1},\theta^{1,0}\rangle =  \left( H \ot \Lambda \theta \right) \ot \Lambda ( J\theta)
	\]
	where  $H$ is the basic cohomology, $d \theta =0$, and $d (J \theta)= J\omega_0 \in H \ot \Lambda \theta$, \cite{IshKas}.
Now apply Proposition \ref{C3prop} with $B =  H \ot \Lambda \theta$ and $x=J \theta$.
\end{proof}

\begin{rmk}
The operations in a $C_\infty$-algebra are strongly related to Massey products, \cite{BMM}, and one might view Corollary \ref{cor:VC3} as a uniform version of Corollary \ref{VMPvanish}. As far as we know, there is general no implication between these properties (even if one had Corollary \ref{VMPvanish}  without degree restrictions). 
\end{rmk}

Allowing for a moment the case $d\omega=0$ in the Vaisman condition, we have the following suggestive diagram of implications, which includes formality in the K\"ahler case:
\[
\begin{tikzcd}
\text{K\"ahler}\ar[d,Rightarrow]\ar[r,Rightarrow]&dd^c\ar[r,Rightarrow]\ar[d,Rightarrow]& \text{minimal $C_\infty$ model with }m_k=0,~k\geq 3\ar[d,Rightarrow]\\
\text{Vaisman}\ar[r,Rightarrow]\ar[rr,Rightarrow,bend right]&dd^c+3& \text{minimal $C_\infty$ model with }m_k=0,~k\geq 4.
\end{tikzcd}
\]
This suggests that $dd^c$-type conditions are incompatible with having a highly complex homotopy type. In the next section we will see such a statement made concrete.

\section{Rational homotopy obstructions to $d d^c$-type conditions} \label{sec;hmtpyobstr}

In this section we show that the existence of a complex structure satisfying a variant of the $dd^c+3$-condition imposes non-trivial restrictions on the underlying real homotopy type. We begin with the most basic form of the argument, which already has interesting applications, and provide a generalization below.

First, we need some elementary concepts from rational homotopy theory. All cdga's will be concentrated in non-negative degrees and connected, i.e. $\dim A^0=1$. A cdga $(A,d)$ is called minimal if it is free as a graded-commutative algebra, $A=\Lambda V$, and there is a well-ordered basis $\{x_i\}$ of $V$, with $x_i < x_j$ if $deg(x_i) < deg(x_j)$, such that $dx_i$ is a sum of products of lower order generators. A minimal model for a cdga $A$ is map of cdga's, $\psi: \cM\to A$, such that $\cM$ is minimal and $\psi$ is a quasi-isomorphism. A $k$-minimal model for a cdga $A$ is map of cdga's, $\psi: \cM^k \to A$, such that $\cM^k $ is minimal, generated by degrees less than or equal to $k$, and $H^s(\psi)$ is an isomorphism for $s \leq k$ while $H^{k+1}(\psi)$ is injective. Minimal models and $k$-minimal models always exist and are unique up to isomorphism. There is a simple algorithm for their construction \cite{InfComp}. If a $k$-minimal model is already a minimal model for $A$, we call $A$ $k$-minimal. Typical examples of $1$-minimal cdga's are provided by the differential forms on nilmanifolds.

\begin{ques} \label{qu:fili}
	Consider a filiform nilmanifold $M=G/\Gamma$ where $\Gamma$ is a lattice in the simply connected Lie group $G$ associated with the (1-minimal) cdga of left-invariant forms
	\[\Lambda(\eta^1,...,\eta^6)\quad d\eta^1=d\eta^2=0,~d\eta^k=\eta^1\eta^{k-1}\text{ for }k=3,...,6.
	\]
	Like any even-dimensional nilmanifold, $M$ admits an almost complex structure (e.g. put $J\eta^{2k}=\eta^{2k-1}$). It is known that $M$ does not admit left-invariant complex structures \cite{GM}, and it is unkown whether it admits any complex structures. As a possibly simpler question we may ask: Is it possible that $M$ admits a complex structure which has a fixed $E_1$-isomorphism type for the bicomplex of forms $(\cA(M), \del ,\delb)$? For example, is the following bicomplex possible?
	\[
	\cA(M)\simeq_1\img{S000_3}\oplus \img{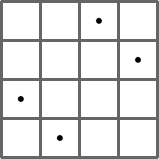}\oplus \img{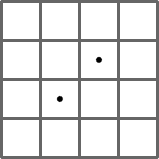}\oplus \img{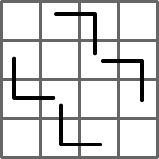}\oplus \img{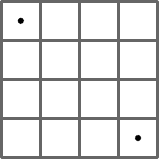}\oplus \img{S311_3}
	\]
	Note that this would yield the correct Betti numbers, satisfy the $dd^c+3$ condition, and have a pure Hodge structure on $H^1$.
\end{ques}

The following Theorem, a prototype for the results in these section, shows that the bicomplex depicted above (and many others)  \textit{cannot} occur as those of a hypothetical complex structure on $M$.

\begin{thm} \label{thm;warmup}
	Let $M$ be a compact manifold of real dimension $2n$ such that 
	\begin{enumerate}
		\item the cdga of forms $(\cA(M),d,\wedge)$ is $1$-minimal.
		\item the cup product map $\cup:H^1(M)\times H^1(M)\to H^2(M)$ vanishes identically.
	\end{enumerate} 
	Assume there exists a complex structure on $M$ such that
	\begin{enumerate}[resume]
		\item the map $H^1(\Ke d^c)\to H^1(M)$ is an isomorphism.
		\item the map $H^2(\Ke d^c)\to H^2(M)\oplus H^2(M)$ is injective.
	\end{enumerate}Then, $n=0$.
\end{thm}
Here, the maps alluded to in $(3)$ and $(4)$ refer to those induced from the left inclusion map $i$, and the 
direct sum map $(i, \II \circ \pi)$ where $ \II=i^{p-q}$, in the following diagram:
\[
\xymatrix{
	& \left( \Ke d^c, d \right) \ar[dl]_i \ar[dr]^{\II \circ\pi} &  \\
	\left(\cA(M),d \right)& & \left(H_{d}(M) , d=0  \right)
}
\]
The example in Question \ref{qu:fili} satisfies conditions $(3)$ and $(4)$ of Theorem \ref{thm;warmup}. These conditions will be discussed in more depth below.

\begin{proof}
Take a $1$-minimal model $\psi: \cM^1(\Ke d^c)=\Lambda V\to \Ke d^c$. We show $i\circ\psi$ is a $1$-minimal model for $\cA(M)$, and therefore a minimal model for $\cA(M)$, since $\cA(M)$ is $1$-minimal. 

Firstly, $H^1(i\circ\psi)$ is an isomorphism by assumption, so it suffices to show that $H^2(i\circ\psi)$ is injective. For the latter, observe that 
$H^k(\II \circ \pi \circ\psi)$ vanishes for $k\geq 2$. Indeed, given any class $\mathfrak{c}\in H^k(M^1(\Ke d^c))$, we may write $\mathfrak{c}=[\sum_i v^1_iv^2_i...v^k_i]$ for some $v_i^j\in V$. Then, using that everything in the image of $\II \circ\pi$ is closed, we compute
\[
H^k(\II \circ\pi\circ \psi) (\mathfrak c)=\sum_i[(\II \circ\pi\circ\psi) (v_i^1)]\cup...\cup[(\II \circ\pi\circ\psi)(v_i^k)]=0.
\]
Since $H^k(i \circ \psi, \II \circ\pi \circ \psi):H^k(\cM^1(\Ke d^c))\to H^k(\cA(M))\oplus H^k(\cA(M))$ is injective, we get that $H^k(i\circ\psi)$ is injective for $k=2$. 

Since $i\circ\psi$ is a minimal model for $\cA(M)$,  $H^{2n}(i\circ\psi)$ is an isomorphism, so the generator of $H^{2n}(\Ke d^c)$ lies in the image of $H^{2n}(\psi)$. On the other hand, we have seen that $\operatorname{rank}H^k(\II \circ\pi\circ\psi)=0$ for all $k\geq 2$. Since for any compact complex $2n$-manifold $\rank H^{2n}(\II \circ\pi)=1$, this implies $2n \leq 1$, so $n=0$.
\end{proof}

In what follows, we generalize this result in several directions. Instead of ruling out certain complex structures only on $M$ itself, we will rule out complex structures on any manifold with the real homotopy type of $M$. Furthermore, we will relax the vanishing of the cup product on first degree cohomology, and drop the hypothesis of $1$-minimality.

\subsection{Additional structure on the real homotopy type of complex manifolds}

We first abstract some homotopy-theoretic properties of the diagram 
\[
\xymatrix{
	& \left( \Ke d^c, d \right) \ar[dl]_i \ar[dr]^{\II \circ\pi} &  \\
	\left(\cA,d \right)& & \left(H_{d}(\cA) , d=0  \right)
}
\]
for compact complex manifolds. All cdgas will be connected and concentrated in non-negative degrees.

Recall that for a map of complexes $\varphi:B\to A$, the cone of $\varphi$ is defined by  
\[
cone(\varphi):= \bigoplus_{n \in \Z} B[-1]^n \oplus A^n \quad \quad d_{cone(\varphi)}(b,a)=(-d_Bb, d_Aa -\varphi(b)),
\]
with $B[-1]^n=B^{n+1}$, and the differential of $B[-1]$ is $-d_B$. The inclusion $\nu : A \to cone(\varphi)$ and projection $\delta : cone(\varphi) \to B[-1]$ given by $\delta(b,a)=-b$ yield an exact sequence 
\[
\xymatrix{
0 \ar[r] & A \ar[r]^-\nu &  {cone(\varphi)} \ar[r]^-\delta & B[-1] \ar[r]& 0
}
\]
whose long exact sequence in cohomology is isomorphic to that of $0 \to B \to A \to Coker \to 0$ in the case of an inclusion
$\varphi: B \to A$.

We will use the following two Lemmas, the proof of which we leave to the reader:

\begin{lem}
	For any map of cdga's $\varphi:B\to A$, $cone(\varphi)$ is a differential graded  module over $B$, via the formula
	\begin{align*}
		B\otimes cone(\varphi)&\longrightarrow cone(\varphi)\\
		(b,(b',a))&\longmapsto ((-1)^{|b|}bb',\varphi(b)a).
	\end{align*}
	 In particular, $H(cone(\varphi))$ is a graded module over $H(B)$ via the same formula.
\end{lem}

\begin{lem} \label{lem;coneModNat}
	For any map of cdga's, $\varphi:B\to A$, the $H(B)$-module structure on $H(cone(\varphi))$ is functorial in the sense that, for a homotopy commutative diagram of cdga's
	\[
	\begin{tikzcd}
		B\ar[r, "\varphi"]\ar[d, "f_B",swap]& A\ar[d,"f_A"]\\
		B'\ar[r, "{\varphi'}"]&A',
	\end{tikzcd}
	\]
    there is a map $H(cone(\varphi))\to H(cone(\varphi'))$ that intertwines the module structures over the respective cohomology algebras, and is an isomorphism if $f_A$ and $f_B$ are quasi-isomorphisms.
\end{lem}
\begin{proof}
First, note that the statement is clear if the diagram strictly commutes. Thus, in the above setting, we obtain an morphism of $H(B)$-modules $H(cone(\varphi))\to H(cone(f_A \circ\varphi))$ and also an morphism $H(cone (\varphi'\circ f_B))\to H(cone(\varphi'))$ which intertwines the $H(B)$, resp $H(B')$ module structures. Next, pick a homotopy $K:B\to A'\langle t, dt\rangle$ s.t. $\epsilon_0\circ K=\varphi'\circ f_B$ and $\epsilon_1\circ K=f_A\circ \varphi$, where $\epsilon_i$ sends $t\mapsto i, dt\mapsto 0$, we obtain isomorphisms of $H(B)$-modules $H(cone(f_A\circ \varphi))\leftarrow H(cone(K))\to H(cone(\varphi'\circ f_B))$.
\end{proof}

In what follows, $A$ will denote a connected cdga over the reals with finite-dimensional cohomology satisfying $2n$-dimensional Poincar\'e duality, i.e. $H^{2n}(A)\cong \R$ and the pairing $H^k(A)\times H^{2n-k}(A)\to H^{2n}(A)$ is non-degenerate.

\begin{defn} \label{defn;dcdiagram}
 A $d^c$-diagram for $A$ is a triple $(B,\varphi_A,\varphi_H)$ of a connected, cohomologically finite dimensional cdga $B$, and cdga maps
 \[
 \begin{tikzcd}
 	&B\ar[ld,"\varphi_A",swap]\ar[rd,"\varphi_H"]&\\
 	A&&H(A)
 \end{tikzcd}
\]
such that the following conditions hold:
\begin{enumerate}
\item \textbf{Symmetry:} The long exact sequences associated with $\varphi_A$ and $\varphi_H$ are isomorphic. 
\item \textbf{Connectivity:} $H^0(\varphi_A)$ is an isomorphism and $H^1(\varphi_A)$ is injective.
	\item \textbf{Duality:} Denoting $\varphi:=(\varphi_A,\varphi_H): B\to A\oplus H(A)$, one has $H^{2n}(cone(\varphi))\cong \R$ and the pairing $H^k(B)\otimes H^{2n-k}(cone(\varphi))\to H^{2n}(cone(\varphi))$ is non-degenerate for every $k$.
\end{enumerate}
\end{defn}

\begin{ex}\label{ex: formal dc diagram}
	Consider a cdga $A$ with minimal model $\varphi_A: \cM_A\to A$. If $A$ is formal, there exists a quasi-isomorphism $\varphi_H: \cM_A\to H(A)$ and the diagram
	\[
	\begin{tikzcd}
		& \cM_A\ar[ld,"\varphi_A",swap]\ar[rd,"\varphi_H"]&\\
		A &&H(A)
	\end{tikzcd}
	\]
	is a $d^c$-diagram. In this case, the $H(\cM_A)$-module structure on $cone(\varphi)$ is isomorphic to the 
	module structure of $H(A)$ on itself, and so the duality isomorphism $H^k(\cM_A) \to H^{2n-k}(cone(\varphi))^\vee$ 
	simply recovers duality of $H(A)$.
\end{ex}

\begin{ex}\label{ex: standard dc diagram}
	For every compact complex manifold $M$, the diagram 
	\[
	\begin{tikzcd}
		&\Ke d^c\ar[ld,"i",swap]\ar[rd,"\II\circ\pi"]&\\
		\cA(M) &&H(M)
	\end{tikzcd}
	\]
 is a $d^c$-diagram, which we will call a \emph{standard $d^c$-diagram}.
\end{ex}

{\small
\begin{figure}

\begin{tabular}{ c| c | c | c | c | c | c  }
zigzag type & length &  $\Ke d^c$  & $H(\Ke d^c )$ &H(A) & $\rank H(i)$ & $\rank H(J \circ \pi)$ \\  \hline 
Dot & $1$ & $\CC$ &  $\CC$  & $\CC$ & 1 & 1\\
Square & NA &    $0 \to \CC \to \CC$ &  $(0,0,0)$ & (0,0) & 0 & 0 \\
Odd Out. & $2m+1$ &  $0 \to \CC^{m+1}$ &  $(0,\CC^{m+1})$ & $(0,\CC)$ & 1 & 1   \\
Odd Inc. & $2m+1$ & $\CC \hookrightarrow \CC^{m}$ & $(0,\CC^{m-1})$ & $(\CC,0)$ & 0 & 0\\
Even Out. & $2m$ &  $0 \to \CC^m$ &   $(0 ,\CC^m)$ & (0,0) & 0 & 0  \\
Even Inc. & $2m$ &   $0 \to \CC^m$ &  $(0 , \CC^m)$ & (0,0) & 0 & 0
\end{tabular}
\captionsetup{justification=centering}
 \caption{Ranks of maps in standard $d^c$-diagram}\label{fig:phiRanks}
\end{figure}
 }

\begin{proof}
Since the two long exact sequences have the same underlying spaces, they are isomorphic if only if the maps $\varphi_A$ and $\varphi_H$ have the same rank on cohomology. So, 
the symmetry condition follows from a case-by-case inspection of every indecomposable bicomplex, for which we refer to Figure \ref{fig:phiRanks}. 

The connectivity property follows from the long exact sequence associated to the short exact sequence
\[
0 \to \Ke \, d^c \to \cA  \to \cA/\Ke \, d^c  \to 0
\]
since $H^0(\cA/\Ke \, d^c)=0$ for any complex manifold.

For the duality property, since $\varphi:=(i,\II \circ\pi):\Ke d^c\to \cA(M) \oplus H(M)$ is injective, the natural projection $\pi: cone(\varphi)\to\operatorname{Coker}(\varphi)$, given by 
\[
\pi(b,a) = \varphi(b) + \left(  \cA(M) \oplus H(M) \right) / \Img( \varphi),
\]
induces an isomorphism $H(\pi) : H(cone(\varphi)) \to H(\operatorname{Coker}(\varphi))$. By Lemma \ref{lem;coneModNat}, the projection $\pi$ is compatible with the $H(\Ke d^c)$-module-structure on cohomologies,
and  $\operatorname{Coker}(\varphi) =  \cA/ \Img \, d^c$, by Theorem \ref{LES}, so the duality property follows from Proposition \ref{Pduality}.
\end{proof}

The importance of the concept of $d^c$-diagram comes from the following observation, which shows it is a property of the real homotopy type of $A$ to admit a $d^c$-diagram with given invariants (e.g. cohomology long exact sequences, pairings, etc.).

\begin{prop}\label{prop: dc homotopy invariance}
Let $f:A\to A'$ be a quasi-isomorphism of cdga's.
\begin{enumerate}
	\item Let $(B,\varphi_A,\varphi_{H})$ be a $d^c$-diagram for $A$. The pushforward diagram $f_*(B,\varphi_A,\varphi_H):=(B,f\circ\varphi_A,H(f)\circ\varphi_H)$ is a $d^c$-diagram for $A'$.
	\item Let $(B',\varphi_{A'},\varphi_H')$ be a $d^c$-diagram for $A'$. Denote by $\psi:B\to B'$ a minimal model for $B'$ and by $\varphi_A:B\to A$ a cdga map such that $f\circ \varphi_A$ is homotopic to $\varphi_{A'}\circ \psi$ (which exists, unique up to homotopy). Then the pullback diagram $f^*(B',\varphi_{A'},\varphi_H'):=(B, \varphi_A,H(f)^{-1}\circ\varphi_H'\circ\psi)$ is a $d^c$-diagram of $A$.
\end{enumerate}
In both cases, there is a (homotopy) commutative diagram with vertical quasi-isomorphisms:
\[
\begin{tikzcd}
	&&B\ar[lld]\ar[rrd]\ar[dd]&&\\
	A\ar[dd]&&&&H(A)\ar[dd]\\
	&&B'\ar[lld]\ar[rrd]&&\\
	A'&&&&H(A')
\end{tikzcd}
\]
In particular, the number $\rank{\varphi_A}=\rank{\varphi_H}$ and the number $\dim \Ke (\varphi_A,\varphi_H)$ are invariant under pullback and pushforward.
\end{prop}

Let us now draw some easy consequences from the definition of $d^c$-diagrams and the long exact sequence in cohomology associated to  $0 \to A\oplus H(A)\to cone(\varphi) \to B[-1] \to 0$, which will highlight common features with a standard $d^c$-diagram.

\begin{prop} \label{prop;dcdiagconsequeces}
	Given a $d^c$-diagram $(B,\varphi_A,\varphi_H)$, the cdga $B$ has the following properties:
	\begin{enumerate}
		\item The cohomology of $B$ is concentrated in degrees $0,...,2n$.
		\item There is an inequality $b_k(B)+b_{2n-k}(B)\geq 2b_k(A)$. 
		\item There is an equality of Euler characteristics $\chi(B)=\chi(A)$.
	\end{enumerate}
Further, denoting by $\psi_A:A\to cone(\varphi)$ and $\psi_H:H(A)\to cone(\varphi)$ the maps induced by the inclusion $A\oplus H(A)\to cone(\varphi)$, we have:
\begin{enumerate}[resume]
		\item\label{property: duality phipsi} The duality pairings induce an isomorphism between the long exact sequences
		\[
		...\longrightarrow H^k(B)\overset{H^k(\varphi_H)}{\longrightarrow} H^k(A)\longrightarrow H^k(cone(\varphi_H))\longrightarrow...
		\]
		and the dual of
		\[
		...\longrightarrow H^{2n-k-1}(cone(\psi_H))\longrightarrow H^{2n-k}(A)\overset{H^{2n-k}(\psi_H)}{\longrightarrow}  H^{2n-k}(cone(\varphi))\longrightarrow...
		\]
		 (and similarly for $\varphi_A,\psi_A$).
	\item\label{property: im=im, A/ker=A/ker}  $cone(\varphi_A) \cong cone(\psi_H)$  and  $cone(\varphi_H) \cong cone(\psi_A)$ (degree preserving)
		\end{enumerate}
\end{prop}

\begin{proof} First,  $H^k(B)=0$ for $k >2n$ by duality, since the cone is non-negatively graded. The second claim follows from the long exact sequence
\[
\xymatrix{
\cdots \ar[r] &  H^k(B) \ar[r]^-\varphi &  H^k(A) \oplus H^k(A) \ar[r] &  H^k(cone(\varphi)) \ar[r] &  H^{k+1}(B) \ar[r]  & \cdots 
}
\]
again by duality $H^k(cone(\varphi))  \cong H^{2n-k}(B)$, using exactness at $H^k(A) \oplus H^k(A)$. 
The third claim follows similarly from exactness and duality, since the Euler characteristic is additive along  long exact sequences.

For the fourth claim, note that the choice of a representative for a fundamental class induces a commutative diagram
\[
\begin{tikzcd}
B\ar[r,"\varphi_H"]\ar[d]&H\ar[d]\\
D_{2n}cone(\varphi) \ar[r,"D_{2n}\psi_H"]&D_{2n}H
\end{tikzcd}
\]
where $D_{2n}$ denotes the dualization functor, defined for any complex $C$ as $(D_{2n}C)^k=(C^{2n-k})^\vee$, with differential given (up to sign) by pullback. Thus, we obtain an isomorphism of the associated long exact sequences involving the cones of $\varphi_H$ and $D_{2n}\psi_H$. 

Finally, for the last claim we note that from the definitions there is a short exact sequence
\[
0\longrightarrow H\overset{\psi_H}{\longrightarrow} cone(\varphi)\longrightarrow cone(\varphi_A)\longrightarrow 0.
\]
Thus, comparing this long exact sequence with that induced by
\[
0\longrightarrow cone(\varphi)\longrightarrow cone(\psi_H)\longrightarrow H[-1]\longrightarrow 0
\]
we see that the natural map $cone(\varphi_A)\to cone(\psi_H)$ has to be a quasi-isomorphism. The case of $\varphi_H$ and $\psi_A$ is analogous.
\end{proof}

\begin{rmk}
 	A $d^c$-diagram $(B,\varphi_A,\varphi_H)$ which is quasi-isomorphic (as in Prop. \ref{prop: dc homotopy invariance}) to a standard $d^c$-diagram has certain additional properties:
	\begin{enumerate}
		\item The odd Betti numbers, $b_{2k+1}(B)$, and the sums of complementary Betti numbers, $b_k(B)+b_{2n-k}(B)$, are all even.
		\item \label{cond: mult Hodge} If $A$ comes from a complex manifold $M$, with $\pdef(M)\leq 1$, then $H(B)$ inherits a multiplicative Hodge structure by Prop. \ref{prop: mult Hodge}.
	\end{enumerate}
\end{rmk}

\subsection{Main result and applications}

The main result in this section gives a topological lower bound on the complexity of the bicomplex of complex structures satisfying a $dd^c$-type condition in low degrees, Theorem \ref{thm;TandAimpliesInequ}. We begin with a Lemma that gives several equivalent formulations of this $dd^c$-type condition.
 
\begin{lem} \label{lem;jsimple}
Let $A$ be a Poincar\'e duality cdga with $d^c$-diagram $(B,\varphi_A,\varphi_H)$ with $\varphi = (\varphi_A,\varphi_H)$. The following conditions are equivalent, for any fixed $j\geq0$.
\begin{enumerate}
\item\label{cond: leq j iso, j+1 inj}  $H^s(\varphi_A)$ is an isomorphism for all $s \leq j$ and $H^{j+1}(\varphi)$ is injective.
\item\label{cond: geq 2n-j iso, 2n-j-1 surj} $H^s(\varphi_A)$ is an isomorphism for all $s\geq 2n-j$ and $H^{2n-j-1}(\psi)$ is surjective, where $\psi=\psi_A+\psi_H$ denotes the map: $A\oplus H(A)\to cone(\varphi)$.
\item\label{cond: everything iso down} All maps $H^s(\varphi_A)$, $H^s(\varphi_H)$, $H^{s}(\psi_A)$, $H^s(\psi_H)$ are isomorphisms for $s\leq j$.
\item\label{cond: everything iso up} All maps $H^s(\varphi_A)$, $H^s(\varphi_H)$, $H^{s}(\psi_A)$, $H^s(\psi_H)$ are isomorphisms for $s\geq 2n-j$.
\end{enumerate}
For a standard $d^c$-diagram $(\Ke\, d^c, i , \II \circ \pi)$ coming from a complex manifold $M$, the above conditions are equivalent to
\begin{enumerate}[resume]
\item\label{cond: FSS + pdef} For degrees $s\leq j$, we have $E_1$-degeneration, $b_s=\sum_{p+q=s}h^{p,q}$,  and pure Hodge structure, $\pdef_s(M)=0$. In degree $j+1$, we have $F^pH^{j+1}(M)\cap\bar{F}^qH^{j+1}(M)=0$ whenever $p+q>j+2$.
\item\label{cond: non-allowed zigzags} For any decomposition of $\cA(M)$ into indecomposables, there are no even zigzags and no odd zigzags of length $\geq 3$ in bidegrees $s,s+1$ for $s\leq j$, except possibly $L$-shaped zigzags in degrees $j,j+1$. 
\end{enumerate}
\end{lem}

We emphasize that the condition ``$H^{j+1}(\varphi)$ is injective'' in \ref{cond: leq j iso, j+1 inj}, cannot be dropped, and is equivalent to the vanishing of the connecting homomorphism $\delta: H^j(cone(\varphi)) \to H^{j+1}(B)$, i.e. the $dd^c+3$-condition (in degree $j+1$) when $B = \Ke \, d^c$.

For clarity, we illustrate condition \ref{cond: non-allowed zigzags} explicitly assuming $j=1$ (in total complex dimension $3$, but the low-degree part is the same in any total dimension). The following zigzags cannot occur:
\[
\img{S1101_3}~\quad ; ~ \quad \img{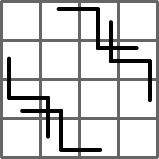}~\quad ; ~ \quad \img{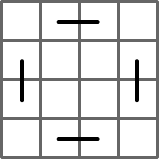} ~\quad ; ~ \quad \img{S100_3} ~\quad ; ~ \quad\img{S222_3},
\]
while, modulo duality, all zigzags in degrees $\geq 2$ are allowed, and the following zigzags in smaller degrees are allowed:
\[
\img{S000_3} ~\quad ; ~ \quad \img{S110_3} ~\quad ; ~ \quad \img{S221_3}
\]

\begin{proof}
The equivalence of condition \ref{cond: everything iso up} and \ref{cond: everything iso down} follows by duality: In fact, by Prop. \ref{prop;dcdiagconsequeces}\ref{property: duality phipsi}, $H^s(\varphi_A)$ (resp. $H^s(\varphi_H)$) is an isomorphism if and only if  $H^{2n-s}(\psi_A)$ (resp. $H^{2n-s}(\psi_H)$) is an isomorphism. 

Next, we show $\ref{cond: leq j iso, j+1 inj}\Rightarrow \ref{cond: everything iso down}$. By the symmetry axiom, $H^s(\varphi_A)$ is an isomorphism if and only if $H^s(\varphi_H)$ is as well. Whenever $H^s(\varphi_A)$ and $H^{s+1}(\varphi_A)$ are isomorphisms,  $H^s(cone (\varphi_A))=0$, so by \ref{prop;dcdiagconsequeces}\ref{property: im=im, A/ker=A/ker}, also $H^s(\psi_H)$ is also an isomorphism. Similarly, whenever $H^s(\varphi_H)$ and $H^{s+1}(\varphi_H)$ are isomorphisms, $H^s(\psi_A)$ is also an isomorphism. Therefore, \ref{cond: leq j iso, j+1 inj} implies $H^s(\varphi_A)$ and $H^s(\varphi_H)$ are isomorphisms for $s\leq j$ and $H^s(\psi_A)$ and $H^s(\psi_H)$ are isomorphisms for $s<j$ and injective for $j=s$. In degree $j=s$, we thus have a diagram:
\[
\begin{tikzcd}
	&H^j(B)\ar[ld,swap,"\simeq"]\ar[rd,"\simeq"]&\\
	H^j(A)\ar[rd, hook]&&H^j(A)\ar[ld,hook']\\
	&H^j(cone(\varphi))
\end{tikzcd}
\]
where we know in addition, from the long exact sequence for $\varphi=(\varphi_A,\varphi_H)$, that the sum of the two bottom maps: $H^j(\psi_A)+H^j(\psi_H):H^j(A)\oplus H^j(A)\to H^j(cone(\varphi))$ is surjective. Then a simple diagram chase yields that both bottom maps are surjective individually. This shows \ref{cond: everything iso down}.

Conversely, if \ref{cond: everything iso down} holds, $H^s(\varphi_A)$ is an iso for $s\leq j$ by assumption and the map $H^j(\psi_A)+H^j(\psi_H)$ has to be surjective since each summand is an isomorphism. Again by the long exact sequence, this implies $H^{j+1}(\varphi)$ is injective, so \ref{cond: leq j iso, j+1 inj} holds. 

The equivalence between \ref{cond: geq 2n-j iso, 2n-j-1 surj} and \ref{cond: everything iso up} follows analogously.

For the equivalence with Condition \ref{cond: non-allowed zigzags}, we refer the reader to Figures \ref{fig:phiRanks}, and \ref{fig:zzddc+3}, and for Condition \ref{cond: FSS + pdef} the origin is \cite[Ch. 2]{StDblBl}, while the argument is a single-degree version of the proof of  Corollary \ref{cor:ddc+3iffE1andpd1}.
\end{proof}

\begin{rmk}
	As this Lemma illustrates, there are many (inequivalent) ways of truncating $dd^c$-type conditions. For instance, in view of conditions \ref{cond: leq j iso, j+1 inj} and \ref{cond: FSS + pdef}, one may call the equivalent conditions of this Lemma as ``$dd^c$ in degrees $\leq j$ and $dd^c+3$ in degree $j+1$''. On the other hand, given condition \ref{cond: everything iso down}, one might call them simply ``$dd^c$ in degrees $\leq j$''. Finally, condition \ref{cond: non-allowed zigzags} suggests neither name would be accurate. To avoid all confusion we choose a neutral name below.
\end{rmk}

\begin{defn}
A $d^c$-diagram $(B,\varphi_A,\varphi_H)$ which satisfies any of the equivalent conditions in Lemma \ref{lem;jsimple}
will be called \emph{j-controlled}. 
A complex structure on a manifold $M$ will be called $j$-controlled if its standard $d^c$-diagram is $j$-controlled.
\end{defn}

Given a $j$-controlled $d^c$-diagram, we can relate the $j$-minimal model of $B$ to the $j$-minimal model of $A$, at least if all cup products into degree $j+1$ are trivial. Namely, for any graded ring $R$, denote by $\langle R^{\leq j}\rangle$ the subring generated in degrees $\leq j$. Then:

\begin{lem}  \label{lem;jminBA}
Fix an integer $j\geq 1$. Let $A$ be a Poincar\'e duality cdga such that
 \[
\langle H^{\leq j}(A) \rangle \cap H^{j+1}(A) = 0.
\]
If $A$ admits a $j$-controlled $d^c$-diagram $(B,\varphi_A,\varphi_H)$, then any $j$-minimal model $\psi:\cM^j \to B$, induces a $j$-minimal model  $\varphi_A \circ \psi: \cM^j \to A$ for $A$.
\end{lem}

\begin{proof}
The map $H^s(\varphi_A\circ\psi)= H^s(\varphi_A) \circ H^s(\psi)$ is an isomorphism for all $s \leq j$ by assumption. 
Also,
\[
H^{j+1}(\varphi \circ\psi) =  (H^{j+1}(\varphi_A \circ\psi ),  H^{j+1}(\varphi_H \circ\psi) )
\] 
is injective by assumption, but the right factor $H^{j+1}(\varphi_H \circ\psi)$ is zero by freeness of the $j$-minimal model for $B$ and the assumption $\langle H^{\leq j}(A) \rangle \cap H^{j+1}(A) = 0$. Therefore, the left factor $H^{j+1}(\varphi_A\circ\psi) $ is injective.
\end{proof}

Now we introduce some topological invariants that will be used in the main theorem below. 

\begin{defn} 
For any cdga $A$ and $k > j \geq 1$, let 
\begin{align*}
r_j^k (A) &= \rank \left(  H^k(\cM^j) \to H^k(A) \right) \\
d_j^k (A) &= \dim \left( \langle H^{\leq j}(A) \rangle \cap H^k(A)  \right),
 \end{align*}
 where $\cM^j \to A$ is a $j$-minimal model of $A$.
 \end{defn}
 
 We note $r_j^k$ is well defined for any $j$ by uniqueness of the $j$-minimal model, up to isomorphism. 
For any $j>k$, $d_j^k (A) \leq r_j^k(A)$, by definition of $j$-minimal model. 
If $A$ is $j$-minimal, then $r_j^k(A)= b_k(A)$, the $k^{th}$ Betti number of $A$, for all $k$. 
If $A$ is $j$-minimal and formal, then $r_j^k(A)= d_j^k(A)$ for all $k$. 

The example of interest is $A = \cA(M)$, and we make the following  observation:
 
 \begin{rmk}  \label{rmk;rkdksum}
 For any $1 \leq j\leq k<n$, the numbers $r_j^k(\cA(M))$ and $d_j^k(\cA(M))$ are both additive with respect to connected sum of $n$-manifolds, namely
\begin{align*}
 r_j^k(\cA(M \# N))) &= r_j^k(\cA(M)) + r_j^k(\cA(N)) \\
  d_j^k(\cA(M \# N))) &= d_j^k(\cA(M)) + d_j^k(\cA(N)).
 \end{align*} 
 For the case $k=n$ we have 
 \begin{align*}
 r_j^{n}(\cA(M \# N))) &= \max\{r_j^{n}(\cA(M)) , r_j^{n}(\cA(N))\} \\
  d_j^{n}(\cA(M \# N))) &= \max\{d_j^{n}(\cA(M)) , d_j^{n}(\cA(N))\}.
 \end{align*} 
 and leftside numbers are either $0$ it $1$, depending on whether the top class is realized in either case. All these follow from the behavior of cohomology rings, and $j$-minimal models, under connected sums.
 \end{rmk}

Next we introduce an `analytic' invariant of $d^c$-diagrams, which the main Theorem will show is bounded below by the non-negative numbers $r_j^k-d_j^k$, under appropriate hypotheses.

\begin{defn}
For any  $d^c$-diagram $(B,\varphi_A,\varphi_H)$ of $A$, let
\[
\ell_k = \dim \Ke H^k(\varphi_A) = \dim \Ke H^k(\varphi_H).
\]
\end{defn}

Note that $\ell_0=\ell_{2n} = 0$ by the definition of $d^c$-diagram,  and that $\ell_k=0$ for $k\leq j$,  for a $j$-controlled $d^c$-diagram, and therefore also $\ell_{2n-k}=0$ for $0 \leq k \leq j$ by the equivalence of 
\ref{cond: leq j iso, j+1 inj} and \ref{cond: geq 2n-j iso, 2n-j-1 surj} in Corollary \ref{lem;jsimple}.
 One can infer from Table \ref{fig:phiRanks} how to compute $\ell_k$ for various zigzag types. The explicit count involves lengths of zigzags and will not be given here, but could be useful in applications.

\begin{thm}  \label{thm;TandAimpliesInequ}
Let $A$ be a Poincar\'e duality cdga. If for some  $j\geq 1$ we have
 \[
\langle H^{\leq j}(A) \rangle \cap H^{j+1}(A) = 0,
\]
and $A$ admits a $j$-controlled $d^c$-diagram $(B,\varphi_A,\varphi_H)$, then
\[
0 \leq  r_j^k -d_j^k   \leq  \ell_k
 \] 
 for all $k > j$.
\end{thm}

In the  inequality above, the term $r_j^k -d_j^k $  is purely topological, and the right hand side is complex-analytic for a standard $d^c$-diagram. In examples below we show the condition $\langle H^{\leq j}(A) \rangle \cap H^{j+1}(A) = 0$ cannot be dropped.

\begin{proof}  Fix $k>j\geq1$, consider a $d^c$-diagram with $j$-minimal model $\psi: \cM^j \to B$ of $B$,
\[
	\begin{tikzcd}
	&\cM^j \ar[d,"\psi",swap] & \\
		&B\ar[ld,"\varphi_A",swap]\ar[rd,"\varphi_H"]&\\
		A &&H(A),
	\end{tikzcd}
	\]
and define
\[
 \nu_j^k  = \rank \left( H^k(\psi): H^k(\cM^j) \to H^k(B) \right) = \dim (V),
 \]
 where $V = \Img \left(H^k(\psi): H^k(\cM^j) \to H^k(B)\right)$.  
 By the assumptions and Lemma \ref{lem;jminBA}, $\psi: \cM^j \to B$, induces a $j$-minimal model  $\varphi_A \circ \psi: \cM^j \to A$ for $A$, therefore for all $k>j$ we have
 \[
 r_j^k = \rank \left(  H^k(\varphi_A \circ \psi ) \right)  \leq  \nu_j^k.
 \]

 The map $H^k(\varphi_H) \big|_V : V \to H^k(A)$ factors through $ \langle H^{\leq j}(A) \rangle \cap H^k(A)$,
since $\cM^j$ is generated by degrees $j$ and lower, and this factoring 
\[
H^k(\varphi_H) \big|_V : V \to  \langle H^{\leq j}(A) \rangle \cap H^k(A)
\]
is surjective since $H^s( \varphi _H \circ \psi)$ is an isomorphism for $s \leq j$, by the assumption that the $d^c$-diagram is $j$-controlled. Then,
\begin{align*}
r_j^k \leq \nu_j^k &= \dim \left( \Img \, H^k(\varphi_H) \big|_V \right) + \dim \left( \Ke H^k(\varphi_H) \big|_V \right) \\
&\leq \dim\left( H^{\leq j}(A) \cap H^k(A) \right) + \dim \left(  \Ke H^k(\varphi_H)  \right) \\
&=d_j^k + \ell_k
\end{align*}
\end{proof}

\begin{rmk}
Under the same hypotheses as the previous Theorem, we have a topological lower bound 
 \[
 2r_j^k-d_j^k \leq b_k(B),
 \]
for all $k> j$ since 
 \[
 b_k(B) = \rank \left(H^k(\varphi_A) \right) + \ell_k \geq r_j^k + \ell_k \geq 2r_j^k-d_j^k.
 \]
 \end{rmk}
 
 Duality gives us the following corollary

\begin{cor} \label{cor;j<kLEQ2n}
Let $A$ be a Poincar\'e duality cdga of formal dimension $2n$. If for some  $j\geq 1$ we have
 \[
\langle H^{\leq j}(A) \rangle \cap H^{j+1}(A) = 0,
\]
and for some $k$, with $2n-j \leq k \leq 2n$, we have $r_j^k > d_j^k$, i.e.
\[
\rank \left( H^k( \cM^j(A) \to A) \right) >  \dim \left( \langle H^{\leq j}(A) \rangle \cap H^k(A)  \right),
\]
then there is no compact complex manifold in the real homotopy type of $A$ with a $j$-controlled complex structure.
 \end{cor}

\begin{proof} If we have a $j$-controlled $d^c$-diagram, then by the equivalences of Lemma \ref{lem;jsimple}, 
$\ell_k=\ell_{2n-k}=0$, so Theorem \ref{thm;TandAimpliesInequ} gives $r_j^k \leq d_j^k$.
\end{proof}

Taking $k=2n$ we recover Theorem \ref{thm;warmup} and Theorem \ref{thm;introhmtpyobst} of the introduction, since then
 $r_j^{2n}=1$ and $d_j^{2n}=0$. We give some examples.

\begin{ex}[Filiform revisited] The real homotopy type of the filiform nilmanifolds of complex dimension $n \geq2$, associated with the cdga of left-invariant forms
	\[
	F_{2n}:=\Lambda (\eta^1,\dots,\eta^{2n}) \qquad d\eta^1=d\eta^2=0,~ d\eta^k=\eta^1\eta^{k-1}\text{ for }k\geq 3,
	\]
	never contains a $1$-controlled complex structure.
	Indeed, nilmanifolds are $1$-minimal, and here the cup product on $H^1$ is trivial, as $d\eta^3 = \eta^1 \eta^2$,
	so that $1 =b_{2n} = r_1^{2n} > d_1^{2n} = 0$.

	As for $k=3$, the filiform nilmanifolds admit almost complex structures (e.g. set $J\eta^{2k}=\eta^{2k-1}$), and are known not to admit left-invariant complex structures in any dimension \cite{GM}. In complex dimension $2$, they are known not to admit any complex structure, as can be recovered here,  since by Corollary \ref{cor:surfaces} any complex structure in complex dimension $2$ is $d d^c+3$, and would be $d d^c$ in degree $1$, as can be shown using  $b_1=2$ is even. It is unknown if this homotopy type can admit a complex structure in dimension $n \geq 3$.
\end{ex}

\begin{ex}[A compact complex $3$-fold satisfying the assumptions of Corollary \ref{cor;j<kLEQ2n}]

Let $M=G/\Gamma$ be a nilmanifold with structure equations 
\begin{align*}
d\eta^3 &= \eta^1 \eta^2  &  d\eta^4 &= \eta^1 \eta^3 \\
 d\eta^5 &= \eta^2 \eta^3  &  d\eta^6 &= \eta^1 \eta^4 + \eta^2  \eta^5.
\end{align*}
Any such nilmanifold has a left invariant complex structure, c.f. \cite{Sal}. 
Then $b_1(M)=2$, and $\eta^1\eta^2 = d\eta^3$, so the product  $\cup:H^1(M)\times H^1(M)\to H^2(M)$ is trivial.  So
$1 =b_{6} = r_1^{6} > d_1^{6} = 0$, and there is no $1$-controlled compact complex manifold with this homotopy type.

According to \cite{COUV} (p. 4, Theorem $2.1$) there are two left-invariant complex structures on $M$. In fact, one may compute that for each of them, the bicomplex looks as follows:
\[
\cA(M)\simeq_1\img{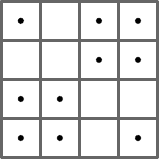}\oplus \img{S311_3}\oplus\img{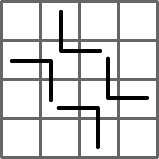}\oplus \img{S1101_3}
\]
So, these satisfy purity in degree $1$, and $\pdef(M)\leq 1$, but there is a nonzero differential $E^{0,1}_2(M)\to E_2^{1,1}(M)$. This is consistent with \cite{COUV}, which shows any left invariant complex structure must degenerate at $E_2$ and not $E_1$, so in particular, is not $dd^c+3$. We emphasize that the results in this case, using the real homotopy type, apply to all complex structures, not only those that are left invariant, and show no complex structure which is $1$-controlled can have this homotopy type.
\end{ex}

\begin{rmk}
From the above examples many others can be constructed by taking connected sum with any manifold $N$ whose first Betti number is zero. Indeed, if $H^1(N)=0$ then
$\langle H^{\leq1}(N)  \rangle \cap H^2(N)=0$ and $r_1^6 = d_1^6=0$ so by Remark \ref{rmk;rkdksum}, $M \# N$ satisfies the topological hypothesis of Corollary \ref{cor;j<kLEQ2n} whenever $M$ does. If both $M$ and $N$ are almost complex, then so is $M \# N$ as well. Almost every orientable $6$-manifold is almost complex, with the only obstruction $W_3 \in H^3(X;\Z)$, so there are a plethora of examples, for which $1$-controlled complex structures are ruled out on $M \# N$. Similar comments apply to blow-ups along almost complex submanifolds and projectivized complex vector bundles.
\end{rmk}

\begin{ex}[The triviality condition on the cup product cannot be dropped] Consider the real homotopy type determined by the cdga with a six dimensional space of generators in degree one and structure equations
	\[
	d\eta^5=\eta^1\eta^3-\eta^2\eta^4,~ d\eta^6=\eta^2\eta^3+\eta^1\eta^4, ~d\eta^i=0\text{ else.}
	\]
 This cdga can be identified with the left-invariant forms on the Iwasawa manifold, given by upper triangular matrices with complex entries modulo those with entries in the Gau{\ss}ian integers. Note that the cup product map $H^1\times H^1\to H^2$ is not trivial. As in every nilmanifold, we have $r_1^{2n}-d_1^{2n}=1$. On the other hand, by construction, the Iwasawa manifold carries a complex structure and some of its small deformations (namely, those of type $(ii.b)$ and $(iii.b)$, according to the classification in \cite{AngellaBook}) are $1$-controlled (this follows from \cite[§9.1]{SteLR}). Thus, the condition on the vanishing cup product in degree $j+1$ cannot be dropped in Theorem \ref{thm;TandAimpliesInequ} or Corollary \ref{cor;j<kLEQ2n}
 
 Also, note this homotopy type has a non-vanishing triple Massey product in $H^2$, so this shows that such Massey products cannot in general be used to rule out the existence of $j$-controlled structures.
\end{ex}

The following example shows Corollary \ref{cor;j<kLEQ2n}  can sometimes be used for $k < 2n$ in situations where $k=2n$ does not apply. 
 
\begin{ex}  \label{ex:k=2,2n=4}
Let $M$ be a $6$-dimensional manifold with the real homotopy type of a nilmanifold with structure equations 
$d\eta^1=d\eta^2=0$,
and
\begin{align*}
d\eta^3 &= \eta^1 \eta^2  &  d\eta^4 &= \eta^2 \eta^3 \\
 d\eta^5 &= \eta^2 \eta^4  &  d\eta^6 &= \eta^1 \eta^5 + \eta^3  \eta^4,
\end{align*}
 c.f. \cite{Mor} or \cite{AngellaBook}. Then 
$H^1 = \langle \eta^1,\eta^2 \rangle$ and  
$H^2 =  \langle \eta^1 \eta^3,\eta^2  \eta^5\rangle$, and 
 $\cup:H^1(M)\times H^2(M)\to H^3(M)$ is trivial, since $d(\eta^1 \eta^4) = \eta^2 \eta^1\eta^3$  and 
$d(\eta^2\eta^6+\eta^3\eta^5) = 2\eta^1\eta^2\eta^5$.  Also, $\cup:H^2(M)\times H^2(M)\to H^4(M)$ is trivial since $d(\eta^1\eta^4\eta^5)=\eta^1\eta^2\eta^3\eta^5$. (The same claims hold with a different underlying homotopy type, changing only the structure equation for $d \eta^5$, to $d\eta^5 = \eta^2 \eta^4-\eta^1\eta^3$. We may use this for $M$ as well.)

So, for any such $M$, we have
\[
2=b_2= r_2^4 > d_2^4 = 0.
\]

Now let $N$ be any orientable $6$-manifold $N$ such that $\cup: H^1 \otimes H^1 \to H^2$ is non-trivial,
$\cup:H^1(N)\times H^2(N)\to H^3(N)$ is trivial, and $d_2^6(N)=1$.  For example, let 
$N = (S^1 \times S^1 \times S^4) \# \CC P^3$.

Now consider $M \# N$. We cannot apply Corollary \ref{cor;j<kLEQ2n} with $j=1$, nor can we apply it with $j=2$ and $k=6$.
But we can apply  Corollary \ref{cor;j<kLEQ2n} to $M \# N$ with $j=2$ and $k=4$, using additivity in 
Remark \ref{rmk;rkdksum}, and conclude $M \# N$  has no complex structure which is $2$-controlled.
\end{ex}

The examples of almost complex manifolds, without $j$-controlled complex structures, are not limited to nilmanifolds and their connected sums with other manifolds. For example, using Milivojevi\'c's realization theorem for almost complex manifolds, one can build examples with very sparse Betti numbers, which are rationally highly connected, by `stretching out' cdga's from the previous examples:

\begin{ex}[Highly connected examples]
Let $s$ be an odd positive integer. Consider the cdga with generators $\eta^1,\eta^2,\eta^3,\eta^4$ in degrees $s,s,2s-1,$ and $3s-2$, respectively, and the only nontrivial relations $d\eta^3 = \eta^1 \eta^2$ and $d \eta^4 = \eta^1 \eta^3$.  The Euler characteristic is zero, the cohomology satisfies Poincar\'e duality, and is trivial in middle degree, so this real homotopy type contains a simply connected $2n:=7s-3$-dimensional almost complex manifold \cite[Thm 2.4., Cor. 6.3. and 6.4.]{Mili}.  Additionally, it is rationally $(s-1)$-connected, $j:=3s-2$-minimal, and satisfies $\langle H^{\leq j}\rangle \cap H^{j+1} = \langle H^{\leq j} \rangle \cap H^{2n} =0$. Indeed, the cohomology $H^s$ vanishes for $s \leq j$, except for $H^s$ generated by $\eta^1$ and $\eta^2$, yet $d\eta^3 = \eta^1 \eta^2$. Then by Corollary \ref{cor;j<kLEQ2n}, any almost complex manifold with this real homotopy type has no complex structure which is $7$-controlled.
\end{ex}

\begin{comment}

\begin{rmk}
	Suppose one would like to show that certain cdga's (say, having the homotopy type of a compact almost complex manifold) do not have the homotopy type of a compact complex manifold. A naive approach could be to enumerate all possible $E_1$-isomorphism types of bicomplexes (i.e. collections of zigzags) which have the same total cohomology and rule them out one by one, e.g. by showing the incompatibility of the corresponding type of $d^c$-diagram with the underlying cdga. One obvious problem with this approach is that there can be an infinite number of such $E_1$-isomorphism types because even length zigzags are not detected by the total cohomology. Thus one may, in first approximation, restrict to those with degenerate Fr\"olicher spectral sequence.\footnote{Since the Hodge numbers behave upper semi-continuously under small deformations, one could for example, very optimistically, hope that if one finds a complex structure in the homotopy type at all, one may find one that has the number $h_{\delb}-b$, and hence the number of even zigzags, below some topological bound.} Then there is only a finite number of possible zigzag-configurations yielding the same Betti numbers. The above result rules out many of those (only dots in degree one, anything with length 1 or 3 in higher degrees) for certain types of algebras.
\end{rmk}
\end{comment}

\bibliographystyle{alpha}

\bibliography{biblio}

\end{document}